\RequirePackage[l2tabu, orthodox]{nag}
\documentclass[reqno]{amsart}
\usepackage[letterpaper,top=2cm,bottom=2cm,left=3cm,right=3cm,
  marginparwidth=2cm]{geometry}

\usepackage{lmodern}
\usepackage[T1]{fontenc}
\usepackage[utf8]{inputenc}
\usepackage[english]{babel}
\usepackage{microtype}
\usepackage{amsmath,amssymb,amsthm,mathrsfs,mathtools,mathdots}
\usepackage{enumerate,booktabs,bm}
\usepackage{tikz}
\usetikzlibrary{decorations.pathmorphing,decorations.pathreplacing,arrows.meta,
  positioning,patterns}
\usepackage{xcolor}
\usepackage{hyperref}

\newtheorem{theorem}{Theorem}[section]
\newtheorem{proposition}[theorem]{Proposition}
\newtheorem{lemma}[theorem]{Lemma}
\newtheorem{corollary}[theorem]{Corollary}

\theoremstyle{definition}
\newtheorem{definition}[theorem]{Definition}
\newtheorem{example}[theorem]{Example}
\newtheorem{remark}[theorem]{Remark}
\newtheorem{question}[theorem]{Question}

\newtheorem*{theoremA}{Theorem A}
\newtheorem*{theoremB}{Theorem B}
\newtheorem*{corollaryC}{Corollary C}

\newcommand{\defin}[1]{%
\relax\ifmmode\textcolor{blue}{#1}%
\else\textcolor{blue}{\emph{#1}}\fi}

\newcommand{\setNN}{\mathbb{N}^{+}}
\newcommand{\setZ}{\mathbb{Z}}

\newcommand{\xvec}{\mathbf{x}}
\newcommand{\avec}{\mathbf{a}}
\newcommand{\elementaryE}{\mathrm{e}}
\newcommand{\LLT}{\operatorname{LLT}}

\newcommand{\Kp}{\mathrm{K}}

\DeclareMathOperator{\asc}{asc}
\DeclareMathOperator{\hrvop}{hrv}
\DeclareMathOperator{\mom}{mom}
\newcommand{\hrv}[2]{\hrvop(#1,#2)}
\newcommand{\Pplus}[1]{P_{+}(#1)}
\newcommand{\ORIENTS}{\mathcal{O}}
\newcommand{\Kv}{\mathsf{K}}
\newcommand{\Ev}{\mathsf{E}}
\newcommand{\Lay}{\mathcal{L}}
\newcommand{\Gplus}{\Gamma^{+}}

\newcommand{\estep}{\texttt{e}}
\newcommand{\dstep}{\texttt{d}}
\newcommand{\nstep}{\texttt{n}}

\definecolor{lightgray}{gray}{0.82}

\tikzset{
  every picture/.append style={baseline=(current bounding box.center)},
  vtx/.style={circle,draw,inner sep=1pt,minimum size=1.3em,font=\small},
  sel/.style={-{Stealth[length=2mm]},thick},
  strictedge/.style={-{Stealth[length=2.4mm]},line width=1.6pt,red!70!black},
  fibre1/.style={fill=yellow},
  fibre2/.style={fill=cyan!30,postaction={pattern=north east lines,
    pattern color=black!55}},
  fibre3/.style={fill=pink!50,postaction={pattern=dots,pattern color=black!60}},
  fibre23/.style={fill=cyan!30,postaction={pattern=crosshatch,
    pattern color=black!55}},
  part1/.style={fill=green!25,postaction={pattern=dots,pattern color=black!60}},
  part2/.style={fill=blue!15,postaction={pattern=north east lines,
    pattern color=black!50}},
  part3/.style={fill=orange!25,postaction={pattern=crosshatch,
    pattern color=black!50}},
  unsel/.style={densely dotted,gray},
}

\pgfdeclarelayer{lltback}
\pgfsetlayers{lltback,main}

\makeatletter
\newif\ifllt@labels
\pgfkeys{/llt/.cd,
  cell/.store in=\llt@cell, cell=1.1em,
  label font/.store in=\llt@labelfont, label font=\scriptsize,
  labels/.is if=llt@labels, labels=true,
  first/.store in=\llt@first, first=1,
  vertex fill/.store in=\llt@vfill, vertex fill=yellow}

\tikzset{
  llt/.style={/llt/.cd,#1,/tikz/.cd,x=\llt@cell,y=\llt@cell},
  llt/.default={},
  llt grid/.style={draw=black!55,line width=0.3pt},
  llt path/.style={line width=1.1pt,line cap=round,line join=round},
  llt arbitrary/.style={line width=1.1pt,dotted,line cap=round},
  llt diagonal/.style={black!45,line width=0.4pt},
  llt bounce/.style={red!75!black,line width=0.9pt,densely dashed,
    -{Stealth[length=1.6mm]}},
  llt asc/.style={-{Stealth[length=1mm]},line width=0.5pt},
  llt cover/.style={-{Stealth[length=2.2mm]},line width=1.4pt,black!45},
  llt dpath/.style={-{Stealth[length=2mm]},thick,densely dashed},
}

\newcount\llt@x
\newcount\llt@y
\newcount\llt@i

\newcommand{\LLTsteps}[4][llt path]{%
  \begingroup
  \pgfmathsetmacro\llt@px{#2}\pgfmathsetmacro\llt@py{#3}%
  \edef\llt@path{(\llt@px,\llt@py)}%
  \@tfor\llt@s:=#4\do{%
    \if n\llt@s\pgfmathsetmacro\llt@py{\llt@py+1}\fi
    \if e\llt@s\pgfmathsetmacro\llt@px{\llt@px+1}\fi
    \if d\llt@s\pgfmathsetmacro\llt@px{\llt@px+1}%
      \pgfmathsetmacro\llt@py{\llt@py+1}\fi
    \edef\llt@path{\llt@path--(\llt@px,\llt@py)}}%
  \draw[#1] \llt@path;%
  \endgroup}

\newcommand{\llt@row}{%
  \llt@i=\llt@x
  \loop\ifnum\llt@i<\llt@y
    \draw[llt grid] (\the\llt@i,\the\llt@y) rectangle +(1,1);%
    \advance\llt@i by 1
  \repeat
  \begin{pgfonlayer}{lltback}
    \fill[\llt@vfill] (\the\llt@y,\the\llt@y) rectangle +(1,1);%
  \end{pgfonlayer}%
  \draw[llt grid] (\the\llt@y,\the\llt@y) rectangle +(1,1);%
  \ifllt@labels
    \node[font=\llt@labelfont,inner sep=0pt]
      at (\the\llt@y+0.5,\the\llt@y+0.5) {$\the\numexpr\llt@y+\llt@first\relax$};%
  \fi}

\newcommand{\LLTdiagram}[1]{%
  \begin{scope}[shift={(\llt@first-1,\llt@first-1)}]
  \begingroup
  \llt@x=0 \llt@y=0
  \@tfor\llt@s:=#1\do{%
    \if n\llt@s\llt@row\advance\llt@y by 1 \fi
    \if d\llt@s\llt@row\advance\llt@x by 1 \advance\llt@y by 1 \fi
    \if e\llt@s\advance\llt@x by 1 \fi}%
  \endgroup
  \LLTsteps{0}{0}{#1}
  \end{scope}}

\newcommand{\LLTframe}[2]{%
  \path (-0.4,-0.4) rectangle (#1+0.4,#2+0.4);%
  \draw[llt diagonal] (-0.4,-0.4)--(#1+0.4,#1+0.4);}

\newcommand{\EdgePic}[4][2em]{%
  \begin{tikzpicture}[x=#1,y=#1,baseline={([yshift=-0.6ex]v1.center)},
      every node/.append style={minimum size=1.05em,font=\scriptsize}]
  \draw[llt arbitrary,line width=0.6pt] (0.4,0)--(#2+0.6,0);
  \foreach \l [count=\i] in {#3}{\node[vtx,fill=white] (v\i) at (\i,0) {$\l$};}
  \foreach \i/\j/\s in {#4}{%
    \pgfmathsetmacro\llt@ang{35+12*(\j-\i)}%
    \def\llt@st{sel}%
    \if\s s\def\llt@st{strictedge}\fi
    \if\s c\def\llt@st{llt cover}\fi
    \if\s p\def\llt@st{llt dpath}\pgfmathsetmacro\llt@ang{\llt@ang+30}\fi
    \if\s b\draw[sel] (v\i) to[out=-\llt@ang,in=180+\llt@ang] (v\j);%
    \else\draw[\llt@st] (v\i) to[out=\llt@ang,in=180-\llt@ang] (v\j);\fi}
  \end{tikzpicture}}

\newcommand{\LLTasc}[1]{%
  \foreach \u/\v in {#1}{\draw[llt asc] (\u-0.82,\v-0.5)--(\u-0.18,\v-0.5);}}

\newcommand{\LLTvertex}[2]{%
  \foreach \v in {#1}{%
    \begin{pgfonlayer}{lltback}
      \fill[#2] (\v-1,\v-1) rectangle (\v,\v);
    \end{pgfonlayer}}}

\newcommand{\LLTshade}[2]{%
  \foreach \u/\v in {#1}{%
    \begin{pgfonlayer}{lltback}
      \fill[#2] (\u-1,\v-1) rectangle (\u,\v);
    \end{pgfonlayer}}}

\newcommand{\LLTmark}[2]{%
  \foreach \u/\v in {#1}{%
    \node[font=\llt@labelfont,inner sep=0pt] at (\u-0.5,\v-0.5) {#2};}}

\newcommand{\LLTpic}[3][]{%
  \begin{tikzpicture}[llt={#1}]\LLTdiagram{#2}#3\end{tikzpicture}}
\makeatother

\newcommand{\FTpic}[2]{%
  \begin{tikzpicture}[x=2.05em,y=2.2em,
      every node/.append style={minimum size=1.15em,font=\scriptsize}]
  \draw[llt arbitrary,line width=0.6pt] (-0.6,0)--(2.6,0);
  \node[vtx,fill=white] (u) at (0,0) {$u$};
  \node[vtx,fill=white] (v) at (1,0) {$v$};
  \node[vtx,fill=white] (w) at (2,0) {$w$};
  \draw[sel] (u) to[out=70,in=110,looseness=1.1] (w);
  \ifnum#1=1 \draw[sel] (u) to[out=50,in=130] (v);\fi
  \ifnum#2=1 \draw[sel] (v) to[out=50,in=130] (w);\fi
  \end{tikzpicture}}

\title[Flip classes of vertical-strip LLT polynomials]{%
Flip classes and elementary expansions\\
of vertical-strip LLT polynomials}

\author{Per Alexandersson}
\address{Department of Mathematics, Stockholm University,
  SE-106 91 Stockholm, Sweden}
\email{per.w.alexandersson@gmail.com}

\author{Lilan Dai}
\address{Center for Combinatorics, LPMC, Nankai University,
  Tianjin 300071, P.~R.~China}
\email{daililan@mail.nankai.edu.cn}


\begin{document}

\begin{abstract}
After the substitution $q\mapsto q+1$, vertical-strip LLT polynomials are
weighted sums of strict $P$-partition generating functions indexed by
orientations of decorated unit-interval graphs, with weight
$q^{\asc(\theta)}$.
A flip moves an ascending edge across a triangle while keeping its long edge.
We show that the sum over every flip class is elementary-positive: it equals the
sum of the elementary symmetric functions given by the highest reachable
vertices of its orientations.
This answers Hwang's question on the Schur positivity of LLT flip-class sums,
and it holds for circular unit-interval graphs as well.
Summing over classes gives a new proof of the elementary expansion of
Alexandersson and Sulzgruber, and its circular analog, which was open.
We also obtain a flip-class refinement of the elementary expansion of
LLT cumulants.
\end{abstract}

\maketitle

\setcounter{tocdepth}{1}
\tableofcontents

\section{Introduction}\label{sec:introduction}

LLT polynomials were introduced by A.~Lascoux, B.~Leclerc and J.-Y.~Thibon as
$q$-deformations of products of Schur functions~\cite{LLT1997}.
I.~Grojnowski and M.~Haiman proved that they are Schur-positive
\cite{GrojnowskiHaiman2006}.
LLT polynomials are the building blocks of the combinatorial formula for
Macdonald polynomials of J.~Haglund, M.~Haiman and
N.~Loehr~\cite{HaglundHaimanLoehr2005}, and of the Shuffle Theorem, conjectured
in~\cite{HaglundHaimanLoehrRemmelUlyanov2005} and proved by E.~Carlsson and
A.~Mellit~\cite{CarlssonMellit2017}; see also~\cite{qtCatalanBook}.

The vertical-strip LLT polynomials form a subfamily which can be described by
colorings of decorated unit-interval graphs~\cite{AlexanderssonPanova2018}: if
$\Gamma$ is such a graph on $[n]$, then
\[
 \LLT_\Gamma(\xvec;q)=\sum_\kappa q^{\asc(\kappa)}x_{\kappa(1)}x_{\kappa(2)}\dotsm x_{\kappa(n)},
\]
where $\kappa$ ranges over the colorings which increase strictly along the
strict edges, and $\asc(\kappa)$ counts the other edges along which $\kappa$
increases.
Without strict edges, they are the unicellular LLT polynomials, which are
related by a plethystic substitution to the chromatic quasisymmetric functions of
J.~Shareshian and M.~Wachs~\cite{ShareshianWachs2012,ShareshianWachs2016},
see~\cite{CarlssonMellit2017,GuayPaquet2016x}.
For chromatic symmetric functions, elementary positivity is the content of the
Stanley--Stembridge conjecture~\cite{StanleyStembridge1993,Stanley1995}.
M.~Guay-Paquet reduced it to unit-interval graphs~\cite{GuayPaquet2013}, and
T.~Hikita proved that each elementary coefficient is nonnegative for every
fixed $q>0$~\cite[Cor.~1.2]{Hikita2026}; the case $q=0$ follows by
specialization.
OpenAI subsequently proved the coefficientwise refinement of Shareshian and
Wachs~\cite{ShareshianWachs2016}: for every natural unit-interval graph, these
coefficients belong to $\mathbb{N}[q]$~\cite[Thm.~1.1]{OpenAI2026}.

The first named author and R.~Sulzgruber proved that, after the substitution
$q\mapsto q+1$, the vertical-strip LLT polynomials expand positively in the
elementary symmetric functions~\cite{AlexanderssonSulzgruber2022}, as conjectured
in~\cite{GarsiaHaglundQiuRomero2019,Alexandersson2019llt}; the positivity was
obtained independently by M.~D'Adderio~\cite{DAdderio2020}.
The expansion is
\begin{equation}\label{eq:introAS}
 \LLT_\Gamma(\xvec;q+1)=\sum_{\theta}q^{\asc(\theta)}\elementaryE_{\lambda(\theta)}(\xvec).
\end{equation}
Here $\theta$ ranges over the orientations of the normal edges of $\Gamma$,
$\asc(\theta)$ is the number of ascending edges, those oriented from the smaller
to the larger vertex, and $\lambda(\theta)$ records the sizes of the sets of vertices with the same highest reachable vertex.
The proof in~\cite{AlexanderssonSulzgruber2022} shows that both sides
of~\eqref{eq:introAS} satisfy a system of linear relations which determines them.
Local linear relations of this kind go back to the modular law of
M.~Guay-Paquet~\cite{GuayPaquet2013} and to the relations of
S.~J.~Lee~\cite{Lee2018x}; see also~\cite{AbreuNigro2020ax}.
The analogous question for circular unit-interval graphs, whose chromatic
quasisymmetric functions were studied by B.~Ellzey~\cite{Ellzey2016,Ellzey2017},
was left open in~\cite[Sec.~6.8]{AlexanderssonSulzgruber2022}.

The same polynomial has a simpler expansion.
Expanding $(1+q)^{\asc(\kappa)}$ shows that
\begin{equation}\label{eq:introK}
 \LLT_\Gamma(\xvec;q+1)=\sum_{\theta}q^{\asc(\theta)}\Kp_{\theta}(\xvec),
\end{equation}
where $\Kp_\theta$ is the generating function of strict $P$-partitions for the
partial order $P=\Pplus{\theta}$ generated by the strict edges and the ascending
edges of $\theta$.
The summands $\Kp_{\theta}$ need not be symmetric functions.
The situation is the same for chromatic symmetric functions, which are sums of
strict $P$-partition generating functions over the acyclic orientations of a
graph~\cite{Stanley1995}.
B.-H.~Hwang introduced flips of acyclic orientations, which reverse a fork of two
covers, and showed that the corresponding sums over flip classes are
symmetric~\cite[Thm.~3.8]{Hwang2022x} and Schur-positive.
He conjectured that they are elementary-positive~\cite[Conj.~3.10]{Hwang2022x}.
He also mentioned an analogous flip on orientations for LLT polynomials,
suggested by the first named author, and asked whether its class sums are
Schur-positive~\cite[Sec.~5.3]{Hwang2022x}.
This is the flip of the present paper.

\subsection*{Main results}

Let $u<v<w$ and let $\theta$ be an orientation in which $uv$ is ascending and
$vw$ is not.
If $uv$ and $uw$ are covers of $\Pplus{\theta}$, and $vw$ and $uw$ are covers
after reversing $uv$ and $vw$, we say that the two orientations differ by a
\emph{flip}, see Definition~\ref{def:flip}.
Thus a flip keeps a long edge $uw$ and moves an ascending short edge across the
triangle.
A flip preserves $\asc(\theta)$ but need not preserve $\Kp_{\theta}$ or
$\lambda(\theta)$.

\begin{theoremA}[Theorem~\ref{thm:main}]
Let $\Gamma$ be an ordinary or circular decorated unit-interval graph, and let
$C$ be a flip class of orientations of $\Gamma$.
Then
\[
 \sum_{\theta\in C}\Kp_{\theta}(\xvec)=\sum_{\theta\in C}\elementaryE_{\lambda(\theta)}(\xvec).
\]
\end{theoremA}

For example, in the complete graph on three vertices the two orientations with
ascending edges $\{12,13\}$ and $\{13,23\}$ form a flip class, and
\[
 \Kp_{\{12,13\}}+\Kp_{\{13,23\}}=\elementaryE_{21}+\elementaryE_3,
\]
see Example~\ref{ex:smallestFlip}.
Thus every class sum is elementary-positive, and in particular Schur-positive,
which answers the question of Hwang mentioned above.
Summing Theorem~A over all flip classes and comparing~\eqref{eq:introK}
with~\eqref{eq:introAS} we obtain the following.

\begin{theoremB}[Corollary~\ref{cor:circularExpansion}]
The expansion~\eqref{eq:introAS} holds for every circular decorated unit-interval
graph.
Here $\theta$ ranges over the admissible orientations, in which the strict and
ascending edges form no directed cycle, and highest reachable vertices are
computed on the lift to $\setZ$.
\end{theoremB}

For ordinary graphs this is~\eqref{eq:introAS} itself, and our proof does not use
the recursion of~\cite{AlexanderssonSulzgruber2022}.
Theorem~A refines~\eqref{eq:introAS}: every flip class contributes an
elementary-positive summand $q^{\asc(C)}\sum_{\theta\in C}\elementaryE_{\lambda(\theta)}$.
The same refinement applies to the LLT cumulants of M.~Dołęga and
M.~Kowalski~\cite{DolegaKowalski2021x}.

\begin{corollaryC}[Corollary~\ref{cor:cumulant}]
The LLT cumulants of circular decorated unit-interval graphs, after
$q\mapsto q+1$, are sums of elementary-positive contributions of flip classes.
\end{corollaryC}

\subsection*{Method}

The proof of Theorem~A for ordinary graphs uses one local identity: for three
vertices $u<v<w$ and any fixed set of other edges containing $uw$, the four ways
of adding the short edges $uv$ and $vw$ satisfy
\[
 \Psi(\text{with } uv)+\Psi(\text{with } vw)=\Psi(\text{with neither})+\Psi(\text{with both}),
\]
both for $\Psi=\Kp$ and for $\Psi=\elementaryE_\lambda$, see Lemma~\ref{lem:fourTerm}.
Summed over all orientations with a fixed number of edges, it gives refined
forms of the two basic bounce relations of~\cite{AlexanderssonSulzgruber2022}.
To pass from all orientations to a single flip class, we prove a connectivity
theorem: in a graph without articulations, all connected orientations with a
given number of ascending edges in which every corner is a cover form one flip
class.
For every other flip class $C'$ of that size, the two class sums $\Kv(C')$ and
$\Ev(C')$ from Theorem~A are expressed, with the same coefficients, as linear
combinations of class sums for graphs with fewer natural edges.
A graph with articulations is cut at the end of its first block, and the
relation is transported to the whole graph by fixing the color of the cut vertex.
The circular case is reduced to the ordinary one: a finite signed expansion
writes every admissible orientation in terms of orientations whose edge lengths
are differences of a potential.  After summing over a flip class, the resulting
coefficients are constant on the flip classes indexing the expansion, and the
flip classes of balanced orientations can be drawn on a line.

Theorem~A for ordinary and circular graphs, and the elementary expansions
obtained from it, have been formally verified in the Lean proof assistant using
Mathlib; the formalization is available at
\url{https://github.com/PerAlexandersson/VerticalStripLLT}.

\subsection*{Outline}

Section~\ref{sec:preliminaries} contains the definitions and the precise
statements, and Section~\ref{sec:local} proves the four-term identity for both
sides of Theorem~A.
Section~\ref{sec:ordinary} uses this identity, several reductions and a gluing
argument to prove the ordinary case, assuming the connectivity theorem proved
in Section~\ref{sec:connectivity}.
That section connects the relevant orientations by reducing to explicit edge
exchanges and complete graphs.
Section~\ref{sec:circular} first expands into balanced orientations, then groups
the expansion by flip classes and reduces each balanced class to ordinary
graphs.
Finally, Section~\ref{sec:consequences} derives the elementary expansion and the
result on cumulants, and ends with two questions.

\section{Preliminaries}\label{sec:preliminaries}

\subsection{Symmetric functions and \texorpdfstring{$P$}{P}-partitions}

We use symmetric functions in $\xvec=(x_1,x_2,\dotsc)$ with coefficients in
$\setZ[q]$, see~\cite{Stanley1972} for the background on $P$-partitions.
For an integer partition $\lambda$ we write
$\elementaryE_\lambda=\elementaryE_{\lambda_1}\elementaryE_{\lambda_2}\dotsm$.
For a map $\kappa$ from a finite set $V$ to $\setNN$ we put
$x_\kappa\coloneqq\prod_{u\in V}x_{\kappa(u)}$.

Let $P$ be a partial order on a finite set $V$.
A map $\kappa:V\to\setNN$ is a \defin{strict $P$-partition} if
$\kappa(u)<\kappa(v)$ whenever $u<_P v$, and we set
\[
 \defin{\Kp(P)}\coloneqq\sum_{\kappa}x_\kappa,
\]
the sum ranging over all strict $P$-partitions.
This is a quasisymmetric function, and it is symmetric when $P$ is a disjoint
union of chains: a chain with $k$ elements contributes $\elementaryE_k$.
We write $u\lessdot_P v$ when $v$ \defin{covers} $u$ in $P$.
If a set of directed edges generates a partial order $P$, then $\Kp(P)$ counts
the maps that increase strictly along every edge, and every cover of $P$ is one
of the edges.

\subsection{Decorated unit-interval graphs}

An \defin{area sequence} is a tuple $\avec=(a_1,\dotsc,a_n)$ of nonnegative
integers with $a_1=0$ and $a_{i+1}\leq a_i+1$.
We put $\defin{b(i)}\coloneqq i-a_i$, so that $b$ is weakly increasing.
A pair $uv$ with $u<v$ is a \defin{natural edge} if $b(v)\leq u$.
The natural edges form a unit-interval graph: if $uw$ is a natural edge and
$u<v<w$, then $uv$ and $vw$ are natural edges.
There are $a_v$ natural edges ending at $v$, so $\sum_i a_i$ natural edges in
total.

A natural edge $uv$ is a \defin{corner} if no other natural edge $u'v'$
satisfies $u'\leq u$ and $v\leq v'$.
Equivalently, $u=b(v)$, and either $v=n$ or $b(v+1)>b(v)$.
Every natural edge lies inside a corner, and distinct corners have distinct
tails and distinct heads.
A \defin{decorated unit-interval graph} $\Gamma=(\avec,S)$ is an area sequence
together with a set $S$ of corners, the \defin{strict edges}.
The other natural edges are \defin{normal}.
Decorated unit-interval graphs are the graphs of Schröder paths
in~\cite[Sec.~2.6]{AlexanderssonSulzgruber2022}: the natural edges are the
cells below the Dyck path obtained by replacing every diagonal step by a north
step followed by an east step, and the strict edges are the cells of the
diagonal steps.

We draw $\Gamma$ as in~\cite{AlexanderssonSulzgruber2022}.
The diagonal box in row $v$ is the vertex $v$, and the cell in column $u$ and
row $v$, for $u<v$, is the pair $uv$.
The natural edges are the cells between the Schröder path and the diagonal, and
the diagonal steps of the path cross the strict edges.

\begin{example}\label{ex:running}
The Schröder path $\mathtt{nnnddeneee}$
of~\cite[Ex.~2.8]{AlexanderssonSulzgruber2022} gives the area sequence
$\avec=(0,1,2,3,3,2)$ with strict edges $14$ and $25$.
The corners are $14$, $25$ and $46$; only the first two are strict.
In the diagram below, arrows mark the ascending edges of an orientation
$\theta$, defined in the next subsection:
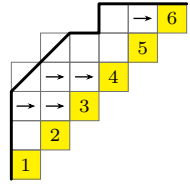
\[
\LLTpic{nnnddeneee}{\LLTasc{1/3,2/3,2/4,3/4,5/6}}
\]
\end{example}

\subsection{Colorings and orientations}

A \defin{coloring} of $\Gamma$ is a map $\kappa:[n]\to\setNN$ with
$\kappa(u)<\kappa(v)$ for every strict edge $uv$.
An \defin{ascent} of $\kappa$ is a normal edge $uv$ with $\kappa(u)<\kappa(v)$,
and $\asc(\kappa)$ is the number of ascents.
The \defin{vertical-strip LLT polynomial} is
\[
 \defin{\LLT_\Gamma(\xvec;q)}\coloneqq\sum_{\kappa}q^{\asc(\kappa)}x_\kappa,
\]
the sum ranging over all colorings of $\Gamma$.
It is a symmetric function~\cite{AlexanderssonPanova2018,HaglundHaimanLoehr2005},
and every vertical-strip LLT polynomial of~\cite{LLT1997} arises in this
way~\cite{AlexanderssonPanova2018,CarlssonMellit2017}.

An \defin{orientation} $\theta$ of $\Gamma$ directs every normal edge.
A normal edge $uv$ with $u<v$ is \defin{ascending} in $\theta$ if it is directed
from $u$ to $v$, and $\asc(\theta)$ is the number of ascending edges.
Strict edges are always directed upwards.
We write $\ORIENTS(\Gamma)$ for the set of orientations, and we identify
$\theta$ with its set of ascending edges.
We regard the set $\theta\cup S$ of ascending and strict edges as a directed graph
on $[n]$.
All its edges point upwards, so its transitive closure is a partial order
$\Pplus{\theta}$ on $[n]$, and we write $\defin{\Kp_\theta}\coloneqq\Kp(\Pplus{\theta})$.

\begin{lemma}\label{lem:stateExpansion}
For every decorated unit-interval graph $\Gamma$,
\[
 \LLT_\Gamma(\xvec;q+1)=\sum_{\theta\in\ORIENTS(\Gamma)}q^{\asc(\theta)}\,
 \Kp_{\theta}(\xvec).
\]
\end{lemma}

\begin{proof}
For a coloring $\kappa$, expand $(q+1)^{\asc(\kappa)}$ as a sum of $q^{|A|}$
over the subsets $A$ of the ascents of $\kappa$.
A pair $(\kappa,A)$ is the same as an orientation $\theta=A$ together with a
strict $\Pplus{\theta}$-partition $\kappa$, since $\kappa$ increases on the
strict edges and on the edges of $A$.
The claim follows.
\end{proof}

For a set $\eta$ of pairs $uv$ with $u<v$, regarded as edges directed upwards,
the \defin{highest reachable vertex} $\hrv{\eta}{u}$ of a vertex $u$ is the
largest vertex which can be reached from $u$ along edges of $\eta$.
For an orientation $\theta$ we write $\hrv{\theta}{u}\coloneqq\hrv{\theta\cup S}{u}$,
the largest $v$ with $u\leq v$ in $\Pplus{\theta}$.
The fibres of $\hrv{\theta}{\cdot}$ form a set partition of $[n]$, and
$\defin{\lambda(\theta)}$ is the partition given by their sizes.

\begin{example}\label{ex:hrv}
We compute the highest reachable vertices of the orientation $\theta$ of
Example~\ref{ex:running}.
The graph $\theta\cup S$ has the strict edges $14$, $25$ and the ascending edges
$13$, $23$, $24$, $34$, $56$.
Following its edges upwards, the vertices reachable from each vertex, and the
largest among them, are
\begin{center}
\begin{tabular}{@{}lcccccc@{}}
\toprule
$u$ & $1$ & $2$ & $3$ & $4$ & $5$ & $6$ \\
\midrule
reachable from $u$ & $1,3,4$ & $2,3,4,5,6$ & $3,4$ & $4$ & $5,6$ & $6$ \\
$\hrv{\theta}{u}$ & $4$ & $6$ & $4$ & $4$ & $6$ & $6$ \\
\bottomrule
\end{tabular}
\end{center}
Note that $\hrv{\theta}{2}=6$ although the edges $23$ and $24$ lead to vertices
whose highest reachable vertex is $4$: the maximum is taken over all directed
paths.
The fibres are $\{1,3,4\}$ and $\{2,5,6\}$, so $\lambda(\theta)=(3,3)$, and
$\theta$ contributes $q^{5}\elementaryE_{33}$ to Theorem~\ref{thm:AS} below,
see also~\cite[Ex.~2.8]{AlexanderssonSulzgruber2022}.
Below, the diagonal boxes and the vertices are shaded by fibre, and the right
picture is the Hasse diagram of $\Pplus{\theta}$:
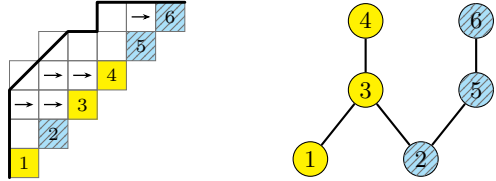
\[
\LLTpic{nnnddeneee}{\LLTasc{1/3,2/3,2/4,3/4,5/6}\LLTvertex{2,5,6}{fibre2}}
\qquad\qquad
\begin{tikzpicture}[x=2.6em,y=2.6em]
\node[vtx,fibre1] (1) at (0,0) {1};
\node[vtx,fibre2] (2) at (1.6,0) {2};
\node[vtx,fibre1] (3) at (0.8,1) {3};
\node[vtx,fibre1] (4) at (0.8,2) {4};
\node[vtx,fibre2] (5) at (2.4,1) {5};
\node[vtx,fibre2] (6) at (2.4,2) {6};
\draw[thick] (1)--(3)--(4) (2)--(3) (2)--(5)--(6);
\end{tikzpicture}
\]
The vertex $2$ lies below both $4$ and $6$, and $\hrv{\theta}{2}$ is the larger of
the two.
The fibres need not be chains of $\Pplus{\theta}$, see
Example~\ref{ex:largerClass}, and $\Kp_{\theta}$ is not symmetric.
\end{example}

The following expansion was conjectured
in~\cite{GarsiaHaglundQiuRomero2019,Alexandersson2019llt}.
Its positivity was obtained independently by M.~D'Adderio~\cite{DAdderio2020}.

\begin{theorem}[{\cite[Cor.~2.10]{AlexanderssonSulzgruber2022}}]
\label{thm:AS}
For every decorated unit-interval graph $\Gamma$,
\[
 \LLT_\Gamma(\xvec;q+1)=\sum_{\theta\in\ORIENTS(\Gamma)}q^{\asc(\theta)}\,
 \elementaryE_{\lambda(\theta)}(\xvec).
\]
\end{theorem}

Comparing with Lemma~\ref{lem:stateExpansion}, the function
$\Kp_{\theta}$ and the elementary function $\elementaryE_{\lambda(\theta)}$
agree after summing over all orientations with a given number of ascending
edges, although they differ for individual orientations.
The purpose of this paper is to identify much smaller sets of orientations over
which they agree.

\subsection{Flips}

\begin{definition}\label{def:flip}
Let $u<v<w$ be vertices such that $uw$ is a natural edge, and let $\theta$ be an
orientation in which $uv$ is ascending and $vw$ is not.
Let $\theta'$ be obtained from $\theta$ by reversing both $uv$ and $vw$, so that
$vw$ is ascending in $\theta'$ and $uv$ is not.
We say that $\theta$ and $\theta'$ differ by a \defin{flip} if
\[
 u\lessdot v \text{ and } u\lessdot w \text{ in } \Pplus{\theta},
 \qquad\text{and}\qquad
 v\lessdot w \text{ and } u\lessdot w \text{ in } \Pplus{\theta'}.
\]
The \defin{flip classes} of $\Gamma$ are the equivalence classes of
$\ORIENTS(\Gamma)$ under the relation generated by flips, see
Figure~\ref{fig:flip}.
\end{definition}

\begin{figure}[!ht]
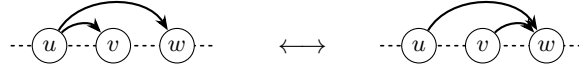

\centering
\EdgePic[2.4em]{3}{u,v,w}{1/3/a,1/2/a}
\qquad$\longleftrightarrow$\qquad
\EdgePic[2.4em]{3}{u,v,w}{1/3/a,2/3/a}
\caption{A flip at $u<v<w$, from $\theta$ (left) to $\theta'$ (right).
The long edge $uw$ is kept, and the ascending short edge moves from $uv$ to
$vw$.
Both drawn edges on each side must be covers, and all other edges of
$\theta\cup S$ are unchanged.}
\label{fig:flip}
\end{figure}

Since $u\lessdot w$ in $\Pplus{\theta}$, the pair $uw$ is an edge of
$\theta\cup S$, the \defin{long edge} of the flip.
The flip retains $uw$, replaces $uv$ by $vw$, and leaves all other edges
unchanged.
The two exchanged \defin{short edges} $uv$ and $vw$ lie strictly inside $uw$;
hence they are not corners, and in particular they are normal.
The definition is symmetric in $\theta$ and $\theta'$, with the roles of the two
short edges exchanged.
Hwang's local flip acts on an acyclic orientation: it reverses the two edges of
a source or sink fork when both are covers in the induced
poset~\cite[Def.~3.4]{Hwang2022x}.
The two outer vertices of his fork are nonadjacent.
Our flip instead retains the long edge $uw$ between the outer vertices and
exchanges the two short edges, with the cover conditions imposed in
$\Pplus{\theta}$ and $\Pplus{\theta'}$.

In an ordinary graph, two of the four cover conditions are automatic.

\begin{lemma}\label{lem:flipCriterion}
Let $\Gamma$ be ordinary, let $u<v<w$ with $uw$ natural, and let $\theta$ be an
orientation in which $uv$ and $uw$ are covers of $\Pplus{\theta}$.
Then $\theta$ and $\theta-uv+vw$ differ by a flip.
Similarly, if $vw$ and $uw$ are covers of $\Pplus{\theta}$, then $\theta$ and
$\theta-vw+uv$ differ by a flip.
\end{lemma}

\begin{proof}
For the first claim, $vw\notin\theta$ since $uw$ is a cover; put
$\theta'=\theta-uv+vw$.
A bypass of $vw$ or of $uw$ in $\theta'\cup S$, drawn dashed, has one of the
forms
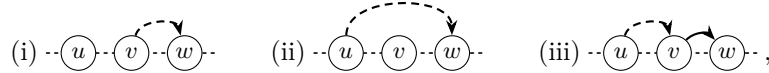
\[
 \text{(i)}\;\EdgePic{3}{u,v,w}{2/3/p}\qquad
 \text{(ii)}\;\EdgePic{3}{u,v,w}{1/3/p}\qquad
 \text{(iii)}\;\EdgePic{3}{u,v,w}{1/2/p,2/3/a}\,,
\]
where in (i) and (ii) the path avoids $vw$.
In (i) the path does not use $uv$, and preceded by $uv$ it bypasses $uw$ in
$\theta\cup S$.
In (ii) it already bypasses $uw$ in $\theta\cup S$.
In (iii) the path from $u$ to $v$ avoids $uv$ and $vw$, so it bypasses $uv$ in
$\theta\cup S$.
Hence $v\lessdot w$ and $u\lessdot w$ in $\Pplus{\theta'}$.
The second claim follows in the same way.
\end{proof}

\begin{example}\label{ex:smallestFlip}
Let $\Gamma$ be the complete graph on $[3]$, with Schröder path
$\mathtt{nnneee}$.
The two orientations
\[
\theta=\;
\LLTpic{nnneee}{\LLTasc{1/2,1/3}}
\qquad\text{and}\qquad
\theta'=\;
\LLTpic{nnneee}{\LLTasc{1/3,2/3}}
\]
differ by a flip at $(1,2,3)$:
the covers of $\Pplus{\theta}$ are $1\lessdot 2$ and $1\lessdot 3$, and those of
$\Pplus{\theta'}$ are $1\lessdot3$ and $2\lessdot3$.
No other flip applies to them, so $C=\{\theta,\theta'\}$ is a flip class.
The vertex $2$ is its own highest reachable vertex in $\theta$, while all
vertices reach $3$ in $\theta'$.
Hence
\[
 \Kp_{\theta}+\Kp_{\theta'}
 =\elementaryE_{21}+\elementaryE_3,
\]
which can be checked on the coefficients of $x_1x_2x_3$ and $x_1^2x_2$.
Neither summand on the left is symmetric.
\end{example}

\begin{lemma}\label{lem:flipInvariants}
On a flip class, the number of ascending edges, the connected components of the
underlying undirected graph of $\theta\cup S$, and the set of corners which
are edges of $\theta\cup S$ are constant.
\end{lemma}

\begin{proof}
A flip exchanges one ascending edge for another and retains the long edge $uw$.
Both short edges join $v$ to the component containing $u$ and $w$.
A corner is not contained in a larger natural edge, so it is never a short edge
of a flip.
\end{proof}

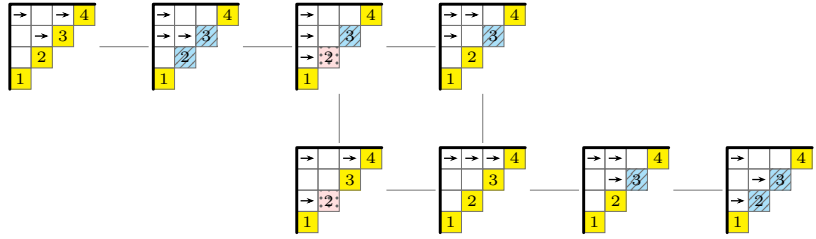
\begin{figure}[!ht]
\centering
\begin{tikzpicture}[x=5.4em,y=5.4em]
\node[inner sep=1pt] (s1) at (2, 0) {\LLTpic[cell=0.8em,label font=\tiny]{nnnneeee}{\LLTasc{1/2,1/3,1/4}\LLTvertex{2}{fibre3}\LLTvertex{3}{fibre2}}};
\node[inner sep=1pt] (s2) at (5, -1) {\LLTpic[cell=0.8em,label font=\tiny]{nnnneeee}{\LLTasc{1/2,2/3,1/4}\LLTvertex{2,3}{fibre2}}};
\node[inner sep=1pt] (s3) at (1, 0) {\LLTpic[cell=0.8em,label font=\tiny]{nnnneeee}{\LLTasc{1/3,2/3,1/4}\LLTvertex{2,3}{fibre2}}};
\node[inner sep=1pt] (s4) at (3, 0) {\LLTpic[cell=0.8em,label font=\tiny]{nnnneeee}{\LLTasc{1/3,1/4,2/4}\LLTvertex{3}{fibre2}}};
\node[inner sep=1pt] (s5) at (4, -1) {\LLTpic[cell=0.8em,label font=\tiny]{nnnneeee}{\LLTasc{2/3,1/4,2/4}\LLTvertex{3}{fibre2}}};
\node[inner sep=1pt] (s6) at (2, -1) {\LLTpic[cell=0.8em,label font=\tiny]{nnnneeee}{\LLTasc{1/2,1/4,3/4}\LLTvertex{2}{fibre3}}};
\node[inner sep=1pt] (s7) at (0, 0) {\LLTpic[cell=0.8em,label font=\tiny]{nnnneeee}{\LLTasc{2/3,1/4,3/4}}};
\node[inner sep=1pt] (s8) at (3, -1) {\LLTpic[cell=0.8em,label font=\tiny]{nnnneeee}{\LLTasc{1/4,2/4,3/4}}};
\draw[gray] (s1) -- (s3);
\draw[gray] (s1) -- (s4);
\draw[gray] (s1) -- (s6);
\draw[gray] (s2) -- (s5);
\draw[gray] (s3) -- (s7);
\draw[gray] (s4) -- (s8);
\draw[gray] (s5) -- (s8);
\draw[gray] (s6) -- (s8);
\end{tikzpicture}
\caption{A flip class of the complete graph on $[4]$, with eight orientations,
each with three ascending edges.
Lines are flips, and the diagonal boxes are shaded by the fibres of
$\hrv{\theta}{\cdot}$.}
\label{fig:largerClass}
\end{figure}

\begin{example}\label{ex:largerClass}
Let $\Gamma$ be the complete graph on $[4]$ without strict edges.
The eight orientations in Figure~\ref{fig:largerClass} form a flip class $C$:
every flip from one of them leads to another one, and the lines show that they
are connected.
Reading off the fibres,
\[
 \Ev(C)=\elementaryE_{211}+2\elementaryE_{22}+3\elementaryE_{31}+2\elementaryE_4,
\]
and by Theorem~\ref{thm:main} this is also $\Kv(C)$.
In the orientation with ascending edges $13$, $14$ and $24$, the fibre
$\{1,2,4\}$ is not a chain of $\Pplus{\theta}$, since $1$ and $2$ are
incomparable.
The $64$ orientations of this graph form $39$ flip classes.
\end{example}

Flip classes can be much larger.
The following table counts the flip classes of all unit-interval graphs on $n$
vertices without strict edges; the largest classes occur for the complete
graphs.
\begin{center}
\begin{tabular}{@{}rrrrrr@{}}
\toprule
$n$ & graphs & orientations & flip classes & classes of size $>1$ &
largest class \\
\midrule
$3$ & $5$ & $17$ & $16$ & $1$ & $2$ \\
$4$ & $14$ & $171$ & $132$ & $25$ & $8$ \\
$5$ & $42$ & $3113$ & $1559$ & $522$ & $84$ \\
$6$ & $132$ & $106419$ & $25276$ & $11331$ & $2589$ \\
\bottomrule
\end{tabular}
\end{center}

\subsection{Circular graphs}\label{subs:circularDefinitions}

A \defin{circular area sequence} is a tuple $\avec=(a_1,\dotsc,a_n)$ with
$n\geq1$, $0\leq a_i\leq n-1$ and $a_{i+1}\leq a_i+1$ for all $i$, with indices
read modulo $n$.
We extend $\avec$ periodically to all of $\setZ$ and put $b(i)=i-a_i$, so that
$b$ is weakly increasing and $b(i+n)=b(i)+n$.
The \defin{lifted natural edges} are the pairs $uv$ of integers with
$b(v)\leq u<v$.
They are invariant under translation by $n$, and their classes modulo $n$ are
the \defin{natural edges} of the \defin{circular unit-interval graph}.
A natural edge from the residue of $u$ to that of $v$ has length $v-u<n$; two
residues may be joined by natural edges in both directions, and we regard these
as different edges.
Corners and strict edges are defined on the lift as before, and give a
\defin{circular decorated unit-interval graph} $\Gamma=(\avec,S)$.
An area sequence with $a_1=0$ has no natural edge ending at $1$, and then
$a_v\leq v-1$ for all $v$; thus the ordinary graphs are precisely the circular
graphs in which no lifted natural edge crosses from one period to the next.
Circular unit-interval graphs and their chromatic quasisymmetric functions were
studied in~\cite{Ellzey2016,Ellzey2017,AlexanderssonPanova2018}.

Colorings $\kappa:\setZ/n\setZ\to\setNN$, ascents and $\LLT_\Gamma(\xvec;q)$ are
defined as before, with one ascent for every normal edge whose tail has a
smaller color than its head.
An orientation $\theta$ is again a set of ascending normal edges, and
$\theta\cup S$ is the directed multigraph on $\setZ/n\setZ$ with the strict
and ascending edges.
Now $\theta\cup S$ may contain directed cycles, for example around the whole
circle.
We call $\theta$ \defin{admissible} if $\theta\cup S$ is acyclic, and we
write $\ORIENTS(\Gamma)$ for the set of admissible orientations.
A directed cycle admits no strictly increasing coloring, so
Lemma~\ref{lem:stateExpansion} holds verbatim, with the sum over admissible
orientations.

Let $\theta$ be admissible and let $\widetilde{\theta}$ be the periodic
graph on $\setZ$ consisting of the lifted strict and ascending edges.
Its edges point upwards, so reachability in $\widetilde{\theta}$ is a
partial order on $\setZ$, and we take covers in this order.
A path in $\widetilde{\theta}$ cannot visit two lifts of the same residue,
since its projection would contain a directed cycle.
Hence every integer $z$ has a largest reachable integer
$\widetilde{\hrvop}(\theta,z)$, and
$\widetilde{\hrvop}(\theta,z+n)=\widetilde{\hrvop}(\theta,z)+n$.
The highest reachable vertex $\hrv{\theta}{u}$ is the residue of
$\widetilde{\hrvop}(\theta,u)$, and $\lambda(\theta)$ is the partition of the
fibre sizes of $\hrv{\theta}{\cdot}$ as before.
For ordinary graphs these definitions agree with the previous ones.

\begin{example}\label{ex:threeCycle}
Let $\avec=(1,1,1)$ with no strict edges.
The natural edges form the directed cycle $1\to2\to3\to1$.
The orientation with ascending edges $3\to1$ and $1\to2$ lifts to the path
$3\to4\to5$, so all three vertices have highest reachable vertex $2$, see
Figure~\ref{fig:lift}.
The admissible orientations are the empty one, the three single edges and the
three paths of two edges, and the right-hand side of
Corollary~\ref{cor:circularExpansion} below is
$\elementaryE_{111}+3q\,\elementaryE_{21}+3q^2\elementaryE_3$.
\end{example}

\begin{figure}[!ht]
\centering
\begin{tikzpicture}[x=3em,y=3em]
\draw[gray] (-0.5,0)--(6.5,0);
\foreach \z in {0,...,6}{\node[vtx,fill=white] (z\z) at (\z,0) {$\z$};}
\draw[sel] (z0) to[bend left=40] (z1);
\draw[sel] (z1) to[bend left=40] (z2);
\draw[sel,line width=1.6pt] (z3) to[bend left=40] (z4);
\draw[sel,line width=1.6pt] (z4) to[bend left=40] (z5);
\draw[sel] (z6) to[bend left=40] (6.6,0.25);
\end{tikzpicture}
\caption{The lift to $\setZ$ of the orientation in Example~\ref{ex:threeCycle}
with ascending edges $3\to1$ and $1\to2$; the integer $z$ lies over the residue
of $z$ modulo $3$.
The path $3\to4\to5$ shows that $3$ reaches the lift $5$ of $2$.}
\label{fig:lift}
\end{figure}

Flips are defined on the lift.
Let $u<v<w<u+n$ be integers such that $uw$ is a lifted natural edge, and let
$\theta,\theta'$ be admissible orientations related as in
Definition~\ref{def:flip}, exchanging the residue classes of $uv$ and $vw$.
They differ by a \defin{flip} if the four cover conditions of
Definition~\ref{def:flip} hold in the reachability orders of
$\widetilde{\theta}$ and $\widetilde{\theta'}$.
Lemma~\ref{lem:flipInvariants} holds for circular graphs as well, with the same
proof.

\subsection{Main results}\label{subs:mainResults}

For a set $\Omega$ of orientations of $\Gamma$ we write
\[
 \defin{\Kv(\Omega)}\coloneqq\sum_{\theta\in\Omega}\Kp_{\theta}(\xvec),
 \qquad
 \defin{\Ev(\Omega)}\coloneqq\sum_{\theta\in\Omega}\elementaryE_{\lambda(\theta)}(\xvec).
\]

\begin{theorem}\label{thm:main}
Let $\Gamma$ be an ordinary or circular decorated unit-interval graph.
Then $\Kv(C)=\Ev(C)$ for every flip class $C$ of $\Gamma$.
\end{theorem}

The number of ascending edges is constant on $C$ by
Lemma~\ref{lem:flipInvariants}; we denote it by $\asc(C)$.
Combining Theorem~\ref{thm:main} with Lemma~\ref{lem:stateExpansion}, every flip
class contributes
\[
 q^{\asc(C)}\Kv(C)=q^{\asc(C)}\sum_{\theta\in C}\elementaryE_{\lambda(\theta)}(\xvec)
\]
to $\LLT_\Gamma(\xvec;q+1)$, which is elementary-positive.
Summing over all classes we obtain the following.

\begin{corollary}\label{cor:circularExpansion}
For every circular decorated unit-interval graph $\Gamma$,
\[
 \LLT_\Gamma(\xvec;q+1)=\sum_{\theta\in\ORIENTS(\Gamma)}q^{\asc(\theta)}
 \elementaryE_{\lambda(\theta)}(\xvec).
\]
In particular, $\LLT_\Gamma(\xvec;q+1)$ is elementary-positive.
\end{corollary}

For ordinary graphs this is Theorem~\ref{thm:AS}.
Our proof does not use the recursion of~\cite{AlexanderssonSulzgruber2022}, and
thus gives a new proof of Theorem~\ref{thm:AS}.

\section{The four-term identity}\label{sec:local}

This section contains the only local identity used in the proof, together with
two facts about highest reachable vertices.
We state them for sets of edges rather than orientations, since later we
apply them to sets obtained by adding or deleting single edges.

Let $\eta$ be a set of pairs $uv$ with $u<v$ in $[n]$, regarded as edges directed
upwards, or a translation-invariant set of such pairs of integers whose
projection to $\setZ/n\setZ$ is acyclic.
In both cases every vertex $z$ has a highest reachable vertex $\hrv{\eta}{z}$;
in the periodic case, the bound follows as in Section~\ref{subs:circularDefinitions}.
The map $\hrv{\eta}{\cdot}$ is weakly decreasing along every edge of $\eta$, and
$\hrv{\eta}{z}$ is the largest sink, that is, vertex without outgoing edges,
which can be reached from $z$.

\begin{lemma}\label{lem:ineffective}
Let $\eta$ be as above.
\begin{enumerate}[(i)]
\item Let $\xi$ be a set of upward pairs, translation invariant in the periodic
case, such that $\hrv{\eta}{y}\leq\hrv{\eta}{x}$ for every $xy\in\xi$.
Then $\hrv{\eta\cup\xi}{\cdot}=\hrv{\eta}{\cdot}$, and in the periodic case
$\eta\cup\xi$ still has acyclic projection.
\item If $\eta$ contains a path from $x$ to $y$ avoiding the edge $xy$, then
deleting $xy$ does not change reachability.
\end{enumerate}
\end{lemma}

\begin{proof}
For (i), $\hrv{\eta}{\cdot}$ is weakly decreasing along every edge of
$\eta\cup\xi$.
Hence every vertex reachable from $z$ in $\eta\cup\xi$ is at most
$\hrv{\eta}{z}$, and the maximizing path in $\eta$ is still available.
In the periodic case, a directed cycle in the projection would lift to a path
from $z$ to $z+kn$ for some $k\geq1$, and then
$\hrv{\eta}{z}+kn=\hrv{\eta}{z+kn}\leq\hrv{\eta}{z}$, which is absurd.
For (ii), every path using $xy$ can be rerouted along the alternative path.
\end{proof}

\begin{lemma}\label{lem:sinkRelabel}
Let $\eta$ and $\eta'$ be sets of upward pairs in $[n]$, and let $\sigma$ be a
permutation of $[n]$ mapping $\eta$ bijectively onto $\eta'$.
If $\sigma$ is increasing on the sinks of $\eta$, then
$\hrv{\eta'}{\sigma(z)}=\sigma(\hrv{\eta}{z})$ for all $z$.
In particular the fibres of $\hrv{\eta'}{\cdot}$ are the images of the fibres of
$\hrv{\eta}{\cdot}$.
\end{lemma}

\begin{proof}
The vertex $\hrv{\eta}{z}$ is the largest sink of $\eta$ reachable from $z$,
and $\sigma$ maps the sinks of $\eta$ reachable from $z$ onto the sinks of
$\eta'$ reachable from $\sigma(z)$, preserving their order.
\end{proof}

We use the Iverson bracket: $[A]=1$ when the statement $A$ holds, and $[A]=0$
otherwise.

\begin{lemma}[The four-term identity]\label{lem:fourTerm}
Let $\eta$ be a set of upward pairs in $[n]$.
Let $u<v<w$ be vertices such that $uw\in\eta$, and $uv,vw\notin\eta$.
Put $\eta_{00}=\eta$, $\eta_{10}=\eta+uv$, $\eta_{01}=\eta+vw$ and
$\eta_{11}=\eta+uv+vw$, and write $\Kp(\eta_{ij})$ for $\Kp$ of the transitive
closure.
\begin{enumerate}[(i)]
\item For every coloring $\kappa:[n]\to\setNN$ which increases along the edges
of $\eta$,
\[
 [\kappa \text{ increases on } \eta_{10}]+[\kappa \text{ increases on } \eta_{01}]
 =1+[\kappa \text{ increases on } \eta_{11}].
\]
In particular $\Kp(\eta_{10})+\Kp(\eta_{01})=\Kp(\eta_{00})+\Kp(\eta_{11})$,
and the same holds for the sums over colorings with prescribed colors at any set
of vertices.
\item The four sets can be paired, either as $\{\eta_{00},\eta_{01}\}$ and
$\{\eta_{10},\eta_{11}\}$ or as $\{\eta_{00},\eta_{10}\}$ and
$\{\eta_{01},\eta_{11}\}$, so that paired sets have the same highest reachable
vertices.
In particular $\Phi(\eta_{10})+\Phi(\eta_{01})=\Phi(\eta_{00})+\Phi(\eta_{11})$
for every function $\Phi$ of the highest reachable vertex map, for example
$\Phi=\elementaryE_{\lambda}$ of its fibre sizes.
\end{enumerate}
\end{lemma}

The lemma has no cover hypotheses, and the other edges of $\eta$ are arbitrary.
Writing each set for its value, the identity reads
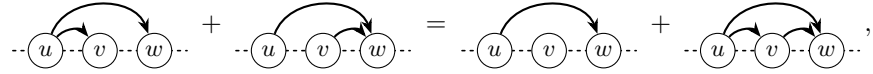
\[
\FTpic{1}{0}\;+\;\FTpic{0}{1}\;=\;\FTpic{0}{0}\;+\;\FTpic{1}{1}\,,
\]
where the dotted line stands for the other vertices and edges, which are the
same in all four graphs.

\begin{proof}
For (i), write $a,b,c$ for the colors at $u,v,w$, so that $a<c$.
Then $[a<b]+[b<c]=1+[a<b<c]$, which is the displayed identity.
Summing $x_\kappa$ over all $\kappa$ increasing on $\eta$, with any colors
prescribed, gives the second statement.

For (ii), write $h(z)=\hrv{\eta}{z}$.
Since $uw\in\eta$, $h(w)\leq h(u)$.
Suppose first that $h(w)\leq h(v)$.
Adding $uv$ does not change the highest reachable vertices of $v$ and $w$, since
no path starting at $v$ or $w$ uses an edge leaving $u<v$.
Hence $\hrv{\eta_{10}}{w}\leq\hrv{\eta_{10}}{v}$ as well, and
Lemma~\ref{lem:ineffective}(i) for $\xi=\{vw\}$ pairs $\eta_{00}$ with $\eta_{01}$
and $\eta_{10}$ with $\eta_{11}$.
Suppose instead that $h(v)<h(w)$.
Then $h(v)<h(w)\leq h(u)$.
In $\eta_{01}$, a path from $v$ either avoids $vw$ or continues from $w$, so
$\hrv{\eta_{01}}{v}\leq\max(h(v),h(w))=h(w)\leq\hrv{\eta_{01}}{u}$.
Lemma~\ref{lem:ineffective}(i) for $\xi=\{uv\}$, applied to $\eta$ and to
$\eta_{01}$, gives the second pairing.
\end{proof}

Lemma~\ref{lem:fourTerm} does not assert that $\eta_{10}$ and $\eta_{01}$ have
the same highest reachable vertex map.
In the first case of the proof, $\eta_{10}$ has the same highest reachable
vertices as $\eta_{11}$, not as $\eta_{01}$.
This is why the identity is used for sums over many orientations at once.

\section{Flip classes of ordinary graphs}\label{sec:ordinary}

In this section we prove Theorem~\ref{thm:main} for ordinary graphs, assuming a
connectivity statement, Theorem~\ref{thm:connectivity}, which we prove in
Section~\ref{sec:connectivity}.

The idea is as follows.
The four-term identity gives linear relations among the sums $\Kv(\Lay_s)$ and
$\Ev(\Lay_s)$ over all orientations with $s$ strict and ascending edges.
These are the bounce relations of~\cite{AlexanderssonSulzgruber2022} in a form
that keeps the number of edges fixed.
To pass from all orientations to one flip class $C$, we show that every other
class with the same number of edges can be written, with the same integer coefficients for
$\Kv$ and $\Ev$, in terms of classes of smaller graphs.
For this we need to know that one class contains all orientations of a certain
kind; this is Theorem~\ref{thm:connectivity}.
The relation is first obtained for an initial piece of the graph, called the
prefix, and then transported to the whole graph by fixing the color of the last
vertex of the prefix.
Throughout, $\Gamma$ is an ordinary decorated unit-interval graph on $[n]$, and
$|E(\Gamma)|=\sum_ia_i$ denotes its number of natural edges.

\subsection{Changing the decoration}

The values $\Kv$ and $\Ev$ only depend on the posets $\Pplus{\theta}$, and the
flip relation only depends on the graphs $\theta\cup S$.
Hence the following two operations transport flip classes to flip classes and
preserve both values.

\begin{lemma}\label{lem:decoration}
Let $C$ be a flip class of $\Gamma$ and let $f$ be a corner.
\begin{enumerate}[(i)]
\item If $f$ is an edge of $\theta\cup S$ for all $\theta\in C$, then marking
$f$ strict, or unmarking it if it is strict, gives a graph $\Gamma'$ in which
$C$ is again a flip class, with the same graphs $\theta\cup S$.
\item If $f$ is an edge of $\theta\cup S$ for no $\theta\in C$, then deleting
$f$ gives a decorated unit-interval graph $\Gamma\setminus f$ with
$|E(\Gamma)|-1$ natural edges, in which $C$ is again a flip class.
\end{enumerate}
\end{lemma}

\begin{proof}
A corner is never a short edge of a flip.
In (i) we therefore only change the label of an edge which is present
throughout, and the flips are the same.
In (ii), deleting the maximal interval $f$ keeps the natural edges closed under
subintervals, and the remaining strict edges are still corners.
No orientation in $C$ uses $f$, so $f$ is not the long edge of a flip from $C$,
and conversely every flip of $\Gamma\setminus f$ is a flip of $\Gamma$.
\end{proof}

By Lemma~\ref{lem:flipInvariants}, every corner is either present throughout
a flip class or absent throughout.

\subsection{Fixing the color of a vertex}

We shall need the two values with the color of one vertex $c$ fixed.
For an orientation $\theta$ and $\beta\in\setNN$, let
$\Kp_\theta[\beta]$ be the sum of $x_\kappa$ over the strict
$\Pplus{\theta}$-partitions $\kappa$ with $\kappa(c)=\beta$.
For the elementary value, replace every fibre of $\hrv{\theta}{\cdot}$ by the
chain on its elements in increasing order.
The strict partitions of this disjoint union of chains give
$\elementaryE_{\lambda(\theta)}$, and we let $\elementaryE_\theta[\beta]$ be the sum over
those with $\kappa(c)=\beta$.
If $c$ is the largest vertex of its fibre, which has $p$ elements, then
\begin{equation}\label{eq:markedE}
 \elementaryE_\theta[\beta]
 =x_\beta\,\elementaryE_{p-1}(x_1,\dotsc,x_{\beta-1})
  \prod_{B}\elementaryE_{|B|}(\xvec),
\end{equation}
the product ranging over the other fibres $B$.
For a set $\Omega$ of orientations we write $\Kv(\Omega)[\beta]$ and $\Ev(\Omega)[\beta]$ for the
corresponding sums.
Summing over $\beta$ gives back $\Kv(\Omega)$ and $\Ev(\Omega)$.

Below, $c$ is always the largest vertex of the graph under consideration, so
that~\eqref{eq:markedE} applies, and $\elementaryE_\theta[\beta]$ only depends on
the fibre sizes and on the size of the fibre of $c$.
The values with the color of $c$ fixed are preserved by every operation used
below: by
Lemma~\ref{lem:decoration}; by deleting a redundant edge, which preserves
$\Pplus{\theta}$; by the four-term identity, whose pairing preserves the whole
map $\hrv{\theta}{\cdot}$; and by relabellings as in Lemma~\ref{lem:sinkRelabel}
which fix $c$.

\subsection{Two relations}

For an integer $s$, let $\defin{\Lay_s(\Gamma)}$ be the set of orientations
$\theta$ such that $\theta\cup S$ has $s$ edges.
By Lemma~\ref{lem:flipInvariants}, $\Lay_s(\Gamma)$ is a union of flip
classes.
We use two local configurations, see Figure~\ref{fig:contacts}.

\begin{definition}\label{def:contacts}
Let $x<y<t$ be vertices with $y=x+1$, such that $xt$ is a strict edge, $y$ is not
the endpoint of any strict edge, and $xt$ is the only strict edge leaving $x$.
Assume that no natural edge joins a vertex at most $y$ to a vertex larger than
$t$, that is, $b(t+1)>y$, where we put $b(n+1)=n+1$.
\begin{enumerate}[(i)]
\item The triple $(x,y,t)$ is an \defin{NN contact} if $b(x)=b(y)$ and no strict
edge ends at $x$.
Its \defin{target} $R$ is obtained by deleting the edge $xt$ and marking the
corner $yt$ strict.
\item Let $u<x$. The quadruple $(u,x,y,t)$ is a \defin{DN contact} if $ux$ is a
strict edge, no other strict edge ends at $x$, and $b(y)=u+1$.
Its \defin{target} $R$ is obtained by changing $a_y$ to $a_y+1$ and $a_t$ to
$a_t-1$, so that $uy$ becomes a natural edge and $xt$ does not, and by replacing
the strict edges $ux,xt$ by $uy,yt$.
\end{enumerate}
\end{definition}

One checks directly from the inequalities that $R$ is a decorated
unit-interval graph.
In the DN case, $a_y=x-u=a_x$, so the new area sequence is valid at $y$, and
$a_{t+1}<a_t$ by the condition on $b(t+1)$.
The target of an NN contact has one natural edge fewer.
The target $R$ of a DN contact has the same number of natural edges as $\Gamma$
and satisfies
\[
 \mom(R)=\mom(\Gamma)-(t-y),
 \qquad \mom(\Gamma)\coloneqq\sum_i i\,a_i.
\]

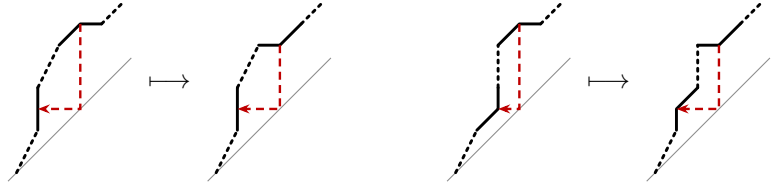
\begin{figure}[!ht]
\centering
\begin{tikzpicture}[llt={cell=0.8em}]
\LLTframe{5}{8}
\draw[llt arbitrary] (0,0)--(1,2) (1,4)--(2,6) (4,7)--(5,8);
\LLTsteps{1}{2}{nn}\LLTsteps{2}{6}{de}
\draw[llt bounce] (3,7)--(3,3)--(1,3);
\end{tikzpicture}
\;$\longmapsto$\;
\begin{tikzpicture}[llt={cell=0.8em}]
\LLTframe{5}{8}
\draw[llt arbitrary] (0,0)--(1,2) (1,4)--(2,6) (4,7)--(5,8);
\LLTsteps{1}{2}{nn}\LLTsteps{2}{6}{ed}
\draw[llt bounce] (3,6)--(3,3)--(1,3);
\end{tikzpicture}
\qquad\qquad
\begin{tikzpicture}[llt={cell=0.8em}]
\LLTframe{5}{8}
\draw[llt arbitrary] (0,0)--(1,2) (2,4)--(2,6) (4,7)--(5,8);
\LLTsteps{1}{2}{dn}\LLTsteps{2}{6}{de}
\draw[llt bounce] (3,7)--(3,3)--(2,3);
\end{tikzpicture}
\;$\longmapsto$\;
\begin{tikzpicture}[llt={cell=0.8em}]
\LLTframe{5}{8}
\draw[llt arbitrary] (0,0)--(1,2) (2,4)--(2,6) (4,7)--(5,8);
\LLTsteps{1}{2}{nd}\LLTsteps{2}{6}{ed}
\draw[llt bounce] (3,6)--(3,3)--(1,3);
\end{tikzpicture}
\caption{An NN contact (left) and a DN contact (right), each followed by its
target, drawn as in~\cite[(3), (4)]{AlexanderssonSulzgruber2022}.
Thick segments are steps of the Schröder path, and dotted segments are
arbitrary.
The diagonal steps at the top cross the strict edge $xt$, and $yt$ in the
targets.
The dashed red path goes down from there to the diagonal box of $x$ and then left to
the path, between the rows of $x$ and $y=x+1$.
In the DN contact the diagonal step at the bottom crosses the strict edge $ux$,
and in its target $uy$.}
\label{fig:contacts}
\end{figure}

\begin{proposition}\label{prop:layer}
Let $\Psi$ be one of $\Kv$, $\Ev$, $\Kv[\beta]$ or $\Ev[\beta]$, with the color
fixed at $c=n$.
If $\Gamma$ has an NN contact with target $R$, then for all $s$
\[
 \Psi(\Lay_s(\Gamma))=\Psi(\Lay_s(R))+\Psi(\Lay_{s-1}(R)).
\]
If $\Gamma$ has a DN contact with target $R$, then $\Psi(\Lay_s(\Gamma))=\Psi(\Lay_s(R))$
for all $s$.
\end{proposition}

\begin{proof}
For an orientation $\theta$ we write $\Psi(\theta)=\Psi(\{\theta\})$.
Comparing coefficients of $q^s$, the two relations are equivalent to
\begin{equation}\label{eq:layerGenerating}
 \sum_{\theta\in\ORIENTS(\Gamma)}q^{|\theta\cup S|}\,\Psi(\theta)
 =(1+q)\sum_{\theta\in\ORIENTS(R)}q^{|\theta\cup S|}\,\Psi(\theta)
\end{equation}
for an NN contact, and to the same identity without the factor $1+q$ for a DN
contact; here $q$ is a formal variable.

Consider an NN contact $(x,y,t)$, and let $\sigma$ be the transposition of $x$
and $y$.
Let $E'$ be the set of natural edges other than $xy$, $yt$ and $xt$.
It is the same for $\Gamma$ and $R$, and $\sigma$ permutes it: for $p<x$, the
pair $px$ is natural if and only if $py$ is, since $b(x)=b(y)$; for $y<p<t$, the
pairs $xp$ and $yp$ are both natural since they lie inside $xt$; and there are no
natural edges from $x$ or $y$ to vertices beyond $t$.
The strict edges in $E'$ avoid $x$ and $y$, so $\sigma$ fixes them.

We group the orientations $\theta$ of $\Gamma$ and of $R$ by the set
$\eta=(\theta\cup S)\cap E'$.
In the pictures below, the vertices $x<y<t$ are drawn in increasing order, the
dotted line stands for $\eta$, thick red arrows are strict edges and thin black
arrows are ascending edges, and each picture stands for the value $\Psi$ of the
corresponding set of edges.
For fixed $\eta$, the orientations of $\Gamma$ contribute $q^{|\eta|}$ times
\[
 q\,\EdgePic{3}{x,y,t}{1/3/s}
 +q^2\,\EdgePic{3}{x,y,t}{1/3/s,1/2/a}
 +q^2\,\EdgePic{3}{x,y,t}{1/3/s,2/3/a}
 +q^3\,\EdgePic{3}{x,y,t}{1/3/s,1/2/a,2/3/a}\,,
\]
and the orientations of $R$ contribute $q^{|\eta|}$ times
\[
 q\,\EdgePic{3}{x,y,t}{2/3/s}
 +q^2\,\EdgePic{3}{x,y,t}{2/3/s,1/2/a}\,.
\]
By the four-term identity, and since $xt$ is redundant once $xy$ and $yt$ are
present,
\[
 \EdgePic{3}{x,y,t}{1/3/s,1/2/a}
 +\EdgePic{3}{x,y,t}{1/3/s,2/3/a}
 =\EdgePic{3}{x,y,t}{1/3/s}
 +\EdgePic{3}{x,y,t}{2/3/s,1/2/a}\,.
\]
Hence the orientations of $\Gamma$ with fixed $\eta$ contribute $(1+q)\,q^{|\eta|}$
times
\[
 q\,\EdgePic{3}{x,y,t}{1/3/s}
 +q^2\,\EdgePic{3}{x,y,t}{2/3/s,1/2/a}\,.
\]
Finally, $\sigma$ maps $\eta+xt$ onto $\sigma(\eta)+yt$.
Its only inversion is the pair $x<y$, and $x$ is not a sink of $\eta+xt$, so
Lemma~\ref{lem:sinkRelabel} applies; $\sigma$ fixes $c=n$, and relabelling
colorings by $\sigma$ preserves $\Kp$.
Hence the first picture for $\eta$ has the same value as the first picture of
$R$ for $\sigma(\eta)$.
Since $\sigma$ permutes the possible sets $\eta$ and preserves their sizes,
summing over $\eta$ gives~\eqref{eq:layerGenerating}.

For a DN contact $(u,x,y,t)$, let $E'$ be the set of natural edges of $\Gamma$ or
$R$ other than $ux$, $uy$, $xy$, $xt$ and $yt$, and again let $\sigma$ exchange
$x$ and $y$.
The previous argument shows that $\sigma$ permutes $E'$, now using that for
$p<x$ with $p\neq u$, the pair $px$ is natural if and only if $py$ is.
With the vertices $u<x<y<t$, the source has the strict edges $ux,xt$ and the
target the strict edges $uy,yt$, and the two remaining local edges of each may be
ascending or not.
We apply the four-term identity to the triangle $x<y<t$ in the source and to the
triangle $u<x<y$ in the target.
It remains to compare the terms with no ascending local edge, and with two:
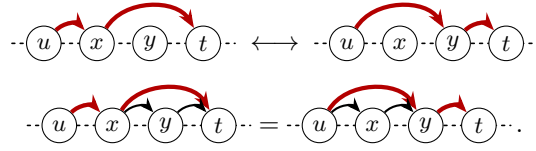
\begin{gather*}
 \EdgePic{4}{u,x,y,t}{1/2/s,2/4/s}
 \;\longleftrightarrow\;
 \EdgePic{4}{u,x,y,t}{1/3/s,3/4/s}
 \\[0.6ex]
 \EdgePic{4}{u,x,y,t}{1/2/s,2/4/s,2/3/a,3/4/a}
 =\EdgePic{4}{u,x,y,t}{1/3/s,3/4/s,1/2/a,2/3/a}\,.
\end{gather*}
In the first line, $\sigma$ maps $\eta+ux+xt$ onto $\sigma(\eta)+uy+yt$, and $x$ is not
a sink, so the values agree after replacing $\eta$ by $\sigma(\eta)$.
In the second line, both sets contain the path $u\to x\to y\to t$, and their remaining
local edges $xt$ and $uy$ are redundant.
As in the NN case, the source and the target then contribute
$(1+q)\,q^{|\eta|+2}$ times the same two values, after reindexing by $\sigma$,
which proves the second relation.
\end{proof}

Multiplying~\eqref{eq:layerGenerating} by $q^{-|S|}$, and using that $\Gamma$ and $R$
have the same number of strict edges, the
NN relation becomes $F(\Gamma)=(1+q)F(R)$ and the DN relation becomes
$F(\Gamma)=F(R)$, for $F(\Gamma)=\LLT_\Gamma(\xvec;q+1)$ and for the right-hand side of
Theorem~\ref{thm:AS}.
These correspond to the shifted forms of the first two bounce relations
of~\cite[Thm.~2.1]{AlexanderssonSulzgruber2022} for a single bounce point.

\begin{remark}\label{rem:fourRelations}
For comparison, we collect the relations
of~\cite[Thm.~2.1 and Thm.~2.5]{AlexanderssonSulzgruber2022} after the
substitution $q\mapsto q+1$.
There $\Gamma$ is encoded by a Schröder path $P$ in the steps
$\nstep,\dstep,\estep$, and we write $F(P)=\LLT_P(\xvec;q+1)$.
Then $F(\nstep\dstep^k\estep)=\elementaryE_{k+1}$, $F$ is multiplicative under
concatenation, and
\begin{align*}
 F(U\nstep\estep V)&=F(U\estep\nstep V)+qF(U\dstep V)
 &&\text{(corner)},\\
 F(U\nstep.\nstep V\dstep.\estep W)&=(1+q)F(U\nstep.\nstep V\estep.\dstep W)
 &&\text{(NN)},\\
 F(U\dstep.\nstep V\dstep.\estep W)&=F(U\nstep.\dstep V\estep.\dstep W)
 &&\text{(DN)},\\
 F(U\nstep.\dstep V\dstep.\estep W)&=qF(U\nstep.\dstep V\estep.\dstep W)
 +(1+q)F(U\dstep.\nstep V\estep.\dstep W)
 &&\text{(ND)},
\end{align*}
where in the last three relations $V\in\{\nstep,\dstep\}^*$ and the bounce path
at the marked points has a single bounce point.
The corner relation deletes a normal corner, or marks it strict; it follows by
summing Lemma~\ref{lem:decoration} over all flip classes, since every corner is
present throughout a flip class or absent from it.
The NN and DN relations are the summed forms of Proposition~\ref{prop:layer}.
We do not use the ND relation: in Lemma~\ref{lem:firstContact}, a strict corner
ending at $y$ is unmarked instead, by Lemma~\ref{lem:decoration}(i).
\end{remark}

\begin{example}\label{ex:NNcontact}
In the graph $\Gamma$ of Example~\ref{ex:running}, with path
$\mathtt{nnnddeneee}$, area sequence $\avec=(0,1,2,3,3,2)$ and strict edges $14$ and
$25$, the triple $(2,3,5)$ is an NN contact: $25$ is the only strict edge leaving
$2$, no strict edge ends at $2$ or $3$, $b(2)=b(3)=1$, and $b(6)=4>3$.
In the notation of Remark~\ref{rem:fourRelations}, the path is
$U\nstep.\nstep V\dstep.\estep W$ with $U=\mathtt{n}$, $V=\mathtt{d}$ and
$W=\mathtt{neee}$.
Its target $R$ has the path $U\nstep.\nstep V\estep.\dstep W=\mathtt{nnndedneee}$,
with area sequence $(0,1,2,3,2,2)$ and strict edges $14$ and $35$:
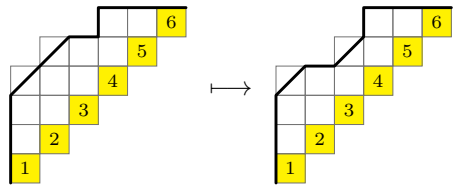
\[
\LLTpic{nnnddeneee}{}
\;\;\longmapsto\;\;
\LLTpic{nnndedneee}{}
\]
The graph $\Gamma$ has nine normal edges and $R$ has eight, and
Proposition~\ref{prop:layer}, summed as above, gives
$\LLT_\Gamma(\xvec;q+1)=(1+q)\LLT_R(\xvec;q+1)$.
\end{example}

\subsection{Three smaller classes}

We now express certain flip classes in terms of flip classes of graphs with
fewer natural edges.
In all three cases the coefficients are the same for $\Kv$ and $\Ev$, and for
the values with the color fixed at $c=n$.

\begin{lemma}\label{lem:absentCorner}
If some corner is absent from $\theta\cup S$ for $\theta\in C$, then $C$ is a
flip class of a graph with fewer natural edges, with the same values.
\end{lemma}

\begin{proof}
This is Lemma~\ref{lem:decoration}(ii).
\end{proof}

\begin{lemma}\label{lem:disconnected}
Suppose that the natural graph of $\Gamma$ is connected, but the underlying
undirected graph of $\theta\cup S$ is disconnected for $\theta\in C$.
Then there is a graph $\Gamma'$ with fewer natural edges on the same vertex
set, and a flip class $C'$ of $\Gamma'$, with the same values, also with the
color fixed at $c=n$.
\end{lemma}

\begin{proof}
The connected components $B_1,\dotsc,B_r$ of $\theta\cup S$ are the same for
all $\theta\in C$.
List them with the component of $n$ last and the others by their least
elements, and let $\rho$ be the permutation of $[n]$ placing them as consecutive
intervals in this order, keeping the order inside each component.
Let $\Gamma'$ be the disjoint union of the graphs induced by $\Gamma$ on the
components, placed on these intervals.
An induced subgraph of a unit-interval graph on a subset, with the inherited
order, is again a unit-interval graph, and its strict edges are still corners.
Since the natural graph of $\Gamma$ is connected, some natural edge joins two
components, and it is lost; hence $|E(\Gamma')|<|E(\Gamma)|$.

Every edge of $\theta\cup S$, every flip, and every directed path between two
vertices of a component lies inside that component.
Hence $\rho$ maps $C$ bijectively onto a flip class $C'$ of $\Gamma'$, and the
cover tests agree.
Relabelling colorings by $\rho$ preserves $\Kp$.
The highest reachable vertex of $z$ lies in the component of $z$, and $\rho$ is
increasing on each component; so $\rho$ maps fibres to fibres, and fixes $n$
together with the position of $n$ in its fibre.
\end{proof}

\begin{lemma}\label{lem:noCover}
Suppose that every corner is strict, and that for every $\theta\in C$ some corner
is not a cover of $\Pplus{\theta}$.
For a set $T$ of corners, let $\Gamma\setminus T$ be the graph with the corners
in $T$ deleted.
Then the orientations in $C$, viewed as orientations of $\Gamma\setminus T$,
form a union of flip classes of $\Gamma\setminus T$, and
\[
 \Psi(C)=\sum_{\emptyset\neq T}(-1)^{|T|+1}\,\Psi(C;\Gamma\setminus T)
\]
for each of the values $\Psi$ above, where $\Psi(C;\Gamma\setminus T)$ is the value
of $C$ as a set of orientations of $\Gamma\setminus T$, and the sum ranges over
nonempty sets of corners.
\end{lemma}

\begin{proof}
Let $e$ be a natural edge and $f\neq e$ a corner.
A directed path between the endpoints of $e$ stays inside the interval of $e$,
and $f$ is not contained in that interval, since $e$ lies inside some corner and
$f$ is maximal.
So deleting corners other than $e$ does not change whether $e$ is a cover.
The orientations of $\Gamma\setminus T$ are the same sets of ascending edges,
and the flips of $\Gamma\setminus T$ are exactly the flips of $\Gamma$ whose
long edge is not in $T$.
Hence the image of $C$ is closed under flips of $\Gamma\setminus T$, that is,
it is a union of flip classes.

Fix $\theta\in C$ and a corner $f$ which is not a cover of $\Pplus{\theta}$.
The alternative path between the endpoints of $f$ lies strictly inside $f$ and
uses no corner.
It survives the deletion of any set of corners.
Thus, whenever $f\notin T$, Lemma~\ref{lem:ineffective}(ii) shows that
$(\theta\cup S)\setminus T$ and
$(\theta\cup S)\setminus(T\cup\{f\})$ have the same reachability relation.
Pairing these sets gives $\sum_T(-1)^{|T|}\Psi((\theta\cup S)\setminus T)=0$, the
sum ranging over all sets of corners.
The corner $f$ may depend on $\theta$, but the identity holds for each $\theta$.
Summing over $C$ and separating the term $T=\emptyset$ proves the claim.
\end{proof}

\subsection{One flip class}

The following theorem is proved in Section~\ref{sec:connectivity}.
A vertex $v$ is an \defin{articulation} of a connected graph if deleting it
disconnects the graph.

\begin{theorem}\label{thm:connectivity}
Suppose that the natural graph of $\Gamma$ consists of a path $1,2,\dotsc,b$
followed by a connected graph on $[b,n]$ without articulations, where the path
may consist of the single vertex $1=b$.
Let $\theta,\theta'$ be orientations with the same number of ascending edges
such that the underlying graphs of $\theta\cup S$ and $\theta'\cup S$ are
connected and every corner is a cover of both $\Pplus{\theta}$ and
$\Pplus{\theta'}$.
Then $\theta$ and $\theta'$ lie in the same flip class.
\end{theorem}

We shall express flip classes in the form
\begin{equation}\label{eq:reduction}
 \Psi(C)=\sum_iz_i\Psi(C_i)\qquad\text{for }\Psi\in\{\Kv,\Ev,\Kv[\beta],\Ev[\beta]\},
\end{equation}
with the color fixed at $c=n$, where the integers $z_i$ do not depend on $\Psi$,
and every $C_i$ is a flip class of a graph $\Gamma_i$ on $[n]$ such that
$(|E(\Gamma_i)|,\mom(\Gamma_i))$ is lexicographically smaller than
$(|E(\Gamma)|,\mom(\Gamma))$.

\begin{proposition}\label{prop:isolation}
Suppose that $\Gamma$ satisfies the hypothesis of Theorem~\ref{thm:connectivity}
and has an NN or a DN contact.
Let $C$ be a flip class containing an orientation $\theta$ such that
$\theta\cup S$ is connected and every corner is a cover of $\Pplus{\theta}$.
Then $C$ has an expansion~\eqref{eq:reduction}.
\end{proposition}

\begin{proof}
Let $s$ be the number of edges of $\theta\cup S$.
Every other flip class $C'\subseteq\Lay_s(\Gamma)$ omits a corner, or has a
disconnected graph, or contains no orientation in which all corners are covers;
otherwise it would contain an orientation as in Theorem~\ref{thm:connectivity},
and hence equal $C$.
In the first two cases Lemmas~\ref{lem:absentCorner} and~\ref{lem:disconnected}
give an expansion~\eqref{eq:reduction} of $C'$.
In the third case, all corners are present in $C'$, and we mark them strict by
Lemma~\ref{lem:decoration}(i) and apply Lemma~\ref{lem:noCover}.
Proposition~\ref{prop:layer} expresses $\Psi(\Lay_s(\Gamma))$ in terms of one or two
of the sets $\Lay_r(R)$ of the target $R$, which are unions of flip classes of
$R$.
Hence
\[
 \Psi(C)=\Psi(\Lay_s(\Gamma))-\sum_{C'\neq C}\Psi(C')
\]
is an expansion~\eqref{eq:reduction}.
\end{proof}

\subsection{Gluing at one vertex}

Suppose that $c\in[n]$ is such that no natural edge joins a vertex smaller
than $c$ to a vertex larger than $c$.
Let $\Gamma_{\leq c}$ and $\Gamma_{\geq c}$ be the graphs induced on $[1,c]$ and
on $[c,n]$; we say that $\Gamma$ is obtained by \defin{gluing} them at $c$.
An orientation of $\Gamma$ is a pair of orientations of the two parts.
A directed path between two vertices of one part stays in that part, and a
triangle $u<v<w$ with $uw$ natural lies in one part.
Hence the flip classes of $\Gamma$ are the products $C_1\times C_2$ of flip
classes of the two parts.

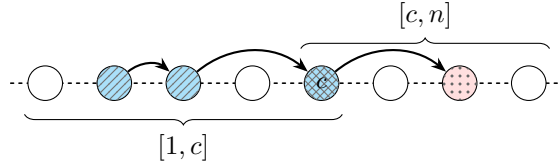
\begin{figure}[!ht]
\centering
\begin{tikzpicture}[x=2.6em,y=2.6em]
\draw[llt arbitrary,line width=0.6pt] (-0.5,0)--(7.5,0);
\foreach \i in {0,3,5,7} \node[vtx,fill=white] (v\i) at (\i,0) {};
\node[vtx,fibre2] (v1) at (1,0) {};
\node[vtx,fibre2] (v2) at (2,0) {};
\node[vtx,fibre23] (v4) at (4,0) {$c$};
\node[vtx,fibre3] (v6) at (6,0) {};
\draw[sel] (v1) to[bend left=40] (v2);
\draw[sel] (v2) to[bend left=40] (v4);
\draw[sel] (v4) to[bend left=40] (v6);
\draw[decorate,decoration={brace,amplitude=4pt,mirror}] (-0.3,-0.45)--(4.3,-0.45)
  node[midway,below=4pt] {$[1,c]$};
\draw[decorate,decoration={brace,amplitude=4pt}] (3.7,0.5)--(7.3,0.5)
  node[midway,above=4pt] {$[c,n]$};
\end{tikzpicture}
\caption{Gluing at $c$.
A left fibre whose highest reachable vertex is $c$ (hatched) merges with the right
fibre containing $c$ (dotted); $c$ lies in both.}
\label{fig:gluing}
\end{figure}

\begin{lemma}\label{lem:gluing}
Let $C_2$ be a flip class of $\Gamma_{\geq c}$, and suppose that, with the
color fixed at $c$,
\[
 \Psi(C_1)[\beta]=\sum_iz_i\Psi(C_{1,i})[\beta]
 \qquad\text{for all $\beta$ and $\Psi\in\{\Kv,\Ev\}$},
\]
where $C_{1,i}$ is a flip
class of a graph $\Gamma_{1,i}$ on $[1,c]$.
Then $\Psi(C_1\times C_2)=\sum_iz_i\Psi(C_{1,i}\times C_2)$ for $\Psi\in\{\Kv,\Ev\}$, where
$C_{1,i}\times C_2$ is a flip class of the graph obtained by gluing
$\Gamma_{1,i}$ and $\Gamma_{\geq c}$ at $c$.
\end{lemma}

\begin{proof}
The glued graphs are decorated unit-interval graphs: a corner of the left part
ends at most at $c$ and a corner of the right part starts at least at $c$, so
neither contains the other.
For orientations $\theta_1$ and $\theta_2$ of the two parts, a strict
$\Pplus{\theta_1\cup\theta_2}$-partition is a pair of strict partitions of the
parts which agree at $c$.
Hence
\[
 \Kp_{\theta_1\cup\theta_2}=\sum_\beta\Kp_{\theta_1}[\beta]\,
 \widehat{\Kp}_{\theta_2}[\beta],
\]
where $\widehat{\Kp}_{\theta_2}[\beta]$ is the sum over strict partitions of
$\Pplus{\theta_2}$ with value $\beta$ at $c$, with the factor $x_\beta$ omitted.

A left vertex reaches the right part only through $c$.
Hence the highest reachable vertex of a left vertex $u$ is
$\hrv{\theta_1}{u}$ if this is smaller than $c$, and $\hrv{\theta_2}{c}$
otherwise, while the right vertices keep their highest reachable vertices.
The fibres containing $c$ in the left and right parts, of sizes $p$ and $r$,
merge into a single fibre of size $p+r-1$; all other fibres are unchanged, see
Figure~\ref{fig:gluing}.
In the two chains, $c$ is the largest element of the first and the smallest
element of the second.
Since
\[
 \sum_\beta x_\beta\,\elementaryE_{p-1}(x_1,\dotsc,x_{\beta-1})\,
 \elementaryE_{r-1}(x_{\beta+1},x_{\beta+2},\dotsc)=\elementaryE_{p+r-1}(\xvec),
\]
we obtain $\elementaryE_{\lambda(\theta_1\cup\theta_2)}
=\sum_\beta\elementaryE_{\theta_1}[\beta]\,\widehat{\elementaryE}_{\theta_2}[\beta]$,
with $\widehat{\elementaryE}_{\theta_2}[\beta]$ defined in the same way.
Multiplying the given identity by $\widehat\Psi(C_2)[\beta]$ and summing over
$\beta$ gives, for $\Psi\in\{\Kv,\Ev\}$,
\begin{align*}
 \Psi(C_1\times C_2)
 &=\sum_{\beta\geq1}\Psi(C_1)[\beta]\,\widehat\Psi(C_2)[\beta] \\
 &=\sum_{\beta\geq1}\sum_i
   z_i\Psi(C_{1,i})[\beta]\,\widehat\Psi(C_2)[\beta] \\
 &=\sum_i z_i\Psi(C_{1,i}\times C_2).
\end{align*}
\end{proof}

If $\Gamma$ is glued at $c$ from parts with $e_1$ and $e_2$ natural edges, then
$\Gamma$ has $e_1+e_2$ natural edges, and
$\mom(\Gamma)=\mom(\Gamma_{\leq c})+\mom(\Gamma_{\geq c})+(c-1)e_2$, since the
area sequence of the right part starts with $a_c=0$.
Hence replacing the left part by a graph which is smaller in
$(|E|,\mom)$ makes the glued graph smaller as well.

\subsection{The prefix}

Suppose that the natural graph of $\Gamma$ is connected, that is, $a_i>0$ for
$i>1$.
If no $a_i$ exceeds one, the natural graph is the path $1,2,\dotsc,n$.

\begin{lemma}\label{lem:prefix}
Suppose that the natural graph of $\Gamma$ is connected and not a path.
Let $j$ be the least index with $a_j\geq2$, put $b=b(j)=j-2$, and let
$c\geq j$ be the least index such that $c=n$ or $a_{c+1}=1$.
Then $\Gamma$ is glued at $c$ from $\Gamma_{\leq c}$ and $\Gamma_{\geq c}$, and
the natural graph of $\Gamma_{\leq c}$ is the path $1,\dotsc,b$ followed by a
graph on $[b,c]$ without articulations.
\end{lemma}

\begin{proof}
A vertex $v$ with $1<v<n$ is an articulation of a connected unit-interval graph
if and only if $a_{v+1}=1$.
Indeed, if $a_{v+1}=1$ then $b(v+1)=(v+1)-a_{v+1}=v$, so no natural edge passes
over $v$.
If $a_{v+1}\geq2$, the natural edge $(v-1)(v+1)$ bypasses $v$ in the path of
edges $i(i+1)$, which are natural since all $a_i$ are positive.
By the choice of $j$ we have $a_i=1$ for $1<i<j$ and $a_j=2$, so $b=b(j)$ and
the vertices $2,\dotsc,b$ are articulations.
For $i>b$ we have $a_i\leq i-b$, so no natural edge ends above $b$ and starts
below $b$.
By the choice of $c$, no vertex strictly between $b$ and $c$ is an
articulation of the graph induced on $[b,c]$, and $b(c+1)=c$ if $c<n$.
\end{proof}

\begin{figure}[!ht]
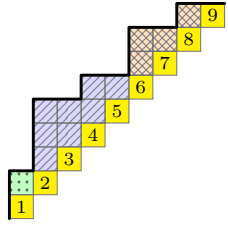

\centering
\LLTpic[cell=0.9em]{nnennneeneenneenee}{\LLTshade{1/2}{part1}\LLTshade{2/3,2/4,3/4,2/5,3/5,4/5,4/6,5/6}{part2}\LLTshade{6/7,6/8,7/8,8/9}{part3}}
\caption{The area sequence $\avec=(0,1,1,2,3,2,1,2,1)$, with $j=4$, $b=2$ and $c=6$ in
Lemma~\ref{lem:prefix}.
The prefix is the path $1,2$ (dotted) followed by the graph on $[2,6]$ without
articulations (hatched), and the whole graph is glued at $6$ from the prefix and
the graph on $[6,9]$ (cross-hatched).}
\label{fig:prefix}
\end{figure}

We call $\Gamma_{\leq c}$ the \defin{prefix} of $\Gamma$, see Figure~\ref{fig:prefix}.
For the prefix we put $a_{c+1}=0$.
The prefix has the same area sequence as $\Gamma$ up to $c$, and
$a_{c+1}<a_c$ both in $\Gamma$ and in the prefix, since $a_c\geq2$.

\begin{lemma}\label{lem:firstContact}
Let $\Gamma$ be a graph in which every corner is strict, with connected natural
graph which is not a path.
Let $t$ be the least index with $a_{t+1}<a_t$, where $a_{n+1}=0$, and put
$x=b(t)$ and $y=x+1$.
Then $a_t\geq2$, and after unmarking the strict edge ending at $y$, if there is
one, $(x,y,t)$ is an NN contact or $(b(x),x,y,t)$ is a DN contact.
\end{lemma}

For the graph of Example~\ref{ex:NNcontact}, $t=5$ and $(2,3,5)$ is an NN
contact.  For the graph in Figure~\ref{fig:prefix}, $t=5$ and $(1,2,3,5)$ is a
DN contact when all corners are marked strict.
The two cases are drawn in Figure~\ref{fig:contacts}.

\begin{proof}
The areas $a_1,\dotsc,a_t$ are weakly increasing.
If $a_t=1$, then $a_i\leq1$ for all $i\leq t$ and $a_{t+1}=0$.
Connectedness forces $t=n$, and the natural graph would then be a path,
contrary to the hypothesis.
Hence $a_t\geq2$ and $x<y<t$.
Since $a_{t+1}<a_t$, the edge $xt$ is a corner, hence strict, and
$b(t+1)\geq x+2>y$.

For $i<t$ we have $b(i+1)-b(i)\in\{0,1\}$, and a corner ends at $i$ exactly when
$b(i+1)=b(i)+1$.
Hence the corners with heads at most $t$ have the tails $1,2,\dotsc,x$, each
exactly once, and the corners with heads above $t$ have tails at least
$b(t+1)\geq x+2$.
In particular, no strict edge leaves $y$, and $xt$ is the only strict edge
leaving $x$.

Since $y<t$, we have $a_x\leq a_y\leq a_x+1$.
If $a_y=a_x+1$, then $b(x)=b(y)$, no corner ends at $x$, and $(x,y,t)$ is an NN
contact once the corner ending at $y$, if any, is unmarked.
If $a_y=a_x$, then $b(y)=b(x)+1$.
The pair $xy$ is natural, so $a_x=a_y\geq1$, and $u=b(x)$ is the tail of the
strict corner $ux$.
Then $(u,x,y,t)$ is a DN contact once the corner ending at $y$, if any, is
unmarked.
\end{proof}

In the terminology of~\cite{AlexanderssonSulzgruber2022}, the four cases are
the contacts NN, ND, DN and DD at the first bounce point.
By Lemma~\ref{lem:decoration}(i), unmarking the corner at $y$ in the ND and DD
cases does not change a flip class in which this corner is present.

\subsection{Proof of Theorem~\ref{thm:main} for ordinary graphs}

We proceed by induction on $n$, then on $|E(\Gamma)|$, then on $\mom(\Gamma)$.
For $n\leq1$ there is only one orientation, and both values are $1$ or
$\elementaryE_1$.

Let $C$ be a flip class of $\Gamma$.
If some corner is absent from $C$, we apply Lemma~\ref{lem:absentCorner} and
induction.
Otherwise we mark all corners strict, by Lemma~\ref{lem:decoration}(i).
If the natural graph is disconnected, it is a disjoint union of graphs on
consecutive intervals; the flip class is a product of flip classes of these
graphs, both values are products, and we use induction on $n$.
If the natural graph is a path, all its edges are strict corners, there is
only one orientation, and both values equal $\elementaryE_n$.
If $\theta\cup S$ is disconnected for $\theta\in C$, or no orientation in $C$
has all corners as covers, we apply Lemma~\ref{lem:disconnected}
or~\ref{lem:noCover} and induction.

In the remaining case, let $\Gamma$ be glued at $c$ from its prefix and
$\Gamma_{\geq c}$ as in Lemma~\ref{lem:prefix}, and write $C=C_1\times C_2$.
Indeed, every flip takes place in one of the two parts, which meet only at $c$.
Let $\theta\in C$ be connected with all corners as covers.
Its restriction to the prefix is connected, and every corner of the prefix is a
corner of $\Gamma$ ending at most at $c$, whose cover test only involves
vertices up to $c$.
By Lemma~\ref{lem:firstContact} and Lemma~\ref{lem:decoration}(i), the prefix
has an NN or DN contact, possibly after unmarking one corner which is present in
$C_1$.
The prefix satisfies the hypothesis of Theorem~\ref{thm:connectivity}, and
Proposition~\ref{prop:isolation} gives an expansion~\eqref{eq:reduction}
$\Psi(C_1)[\beta]=\sum_iz_i\Psi(C_{1,i})[\beta]$ on $[1,c]$, with the color fixed at
the last vertex $c$ of the prefix.
By Lemma~\ref{lem:gluing},
\[
 \Psi(C)=\sum_iz_i\Psi(C_{1,i}\times C_2),\qquad \Psi\in\{\Kv,\Ev\}.
\]
Each $C_{1,i}\times C_2$ is a flip class of a graph on $[n]$ which is smaller in
$(|E|,\mom)$, by the remark after Lemma~\ref{lem:gluing}.
By induction $\Kv(C_{1,i}\times C_2)=\Ev(C_{1,i}\times C_2)$ for all $i$.
Since the same coefficients occur on both sides,
\[
 \Kv(C)=\sum_i z_i\Kv(C_{1,i}\times C_2)
 =\sum_i z_i\Ev(C_{1,i}\times C_2)=\Ev(C).
\]
This completes the proof of Theorem~\ref{thm:main} for ordinary graphs, given
Theorem~\ref{thm:connectivity}.

\section{Connecting orientations by flips}\label{sec:connectivity}

In this section we prove Theorem~\ref{thm:connectivity}.
We first move the edges of an orientation into a normal form determined by the
corners, Section~\ref{subs:preparation}.
Section~\ref{subs:slides} constructs explicit paths of flips which exchange two
edges and restore all others.
With these we treat complete graphs and graphs with at most two corner chains,
and then reduce the general case to smaller graphs by removing part of the
last clique or the first vertex.

Throughout, $\Gamma$ is an ordinary decorated unit-interval graph, and we only
consider orientations in which every corner is a cover.
In particular every corner is an edge of $\theta\cup S$, and by
Lemma~\ref{lem:decoration}(i) we may assume that all corners are strict.
We write $uv\in\theta$ if $uv$ is an edge of $\theta\cup S$, and covers always
refer to $\Pplus{\theta}$.
We write
\[
 vy\mapsto yh\quad[vh]
\]
for the flip which replaces the edge $vy$ by $yh$ and retains the long edge
$vh$, and similarly when the short edges are exchanged in the other direction.
By Lemma~\ref{lem:flipCriterion}, such an exchange is a flip as soon as the two
edges $vy$ and $vh$ are covers before it, and then $vh$ and $yh$ are covers
after it.
We repeatedly use the following observation: a directed path in $\theta\cup S$
between two vertices of an interval stays in that interval, so the cover test
of an edge $uv$ only involves edges inside $[u,v]$.

\subsection{Corner chains}\label{subs:preparation}

Distinct corners have distinct tails and distinct heads, so the corners form
disjoint directed paths, the \defin{corner chains}; a vertex on no corner is a
chain by itself.
Let $\defin{r(u)}$ be the last vertex of the corner chain containing $u$.
A natural edge $uv$ with $u$ and $v$ on the same chain is a corner: if $uh$ is
the corner leaving $u$, then $v\leq h$ by maximality of $uh$, and $v\geq h$
since $v$ comes after $u$ on the chain.
Let $\Gplus$ be the set of corners together with the normal edges $uv$ with
$r(u)<r(v)$.

If $\theta\subseteq\Gplus$, that is, every edge of $\theta\cup S$ lies in
$\Gplus$, then every corner is a cover.
Indeed, $r$ is constant along corners and increases along the other edges of
$\Gplus$, so a path between the endpoints $u,h$ of a corner only uses corners,
and the only corner leaving $u$ is $uh$.

\begin{lemma}\label{lem:preparation}
Every orientation $\theta$ in which all corners are covers is joined by flips to
an orientation $\theta'\subseteq\Gplus$ which contains every edge of
$\theta\cap\Gplus$.
\end{lemma}

\begin{proof}
We process the corner chains in decreasing order of their last vertices.
Let $U$ be the union of the chains already processed.
We maintain that no edge of $\theta$ leaves $U$, and that all corners are covers
when we start a new chain; initially $U=\emptyset$.

Let $B$ be the next chain, with last vertex $t$.
Consider the edges $vy\in\theta$ with $v\in B$ and $r(y)<t$.
For such an edge, $v<y\leq r(y)<t$, so $v$ has a corner successor $h$ on $B$, and
$y<h$ since $vy$ does not contain the corner $vh$.
Choose a vertex $v$ with such an edge, and among its edges of this kind the edge
$vy$ with the least head.
Then $vy$ is a cover.
Indeed, a path from $v$ to $y$ cannot pass through $U$, since no edge leaves
$U$.
An edge from $v$ to a vertex $z\notin U$ has $r(z)\leq t$, by the choice of $t$.
If $r(z)=t$, it is the corner $vh$, and $h>y$; if $r(z)<t$, it is one of the
edges considered, and
$z\geq y$ by the choice of $vy$.
So the only path from $v$ to $y$ is the edge $vy$.
Since $vh$ is a cover as well, we can perform the flip $vy\mapsto yh\ [vh]$,
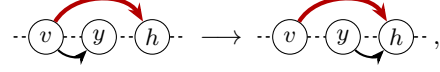
\[
 \EdgePic{3}{v,y,h}{1/3/s,1/2/b}\;\longrightarrow\;\EdgePic{3}{v,y,h}{1/3/s,2/3/b}\,,
\]
after which $vh$ is still a cover.
The new edge $yh$ lies inside the corner $vh$, so it is normal, and
$r(y)<r(h)=t$, so it lies in $\Gplus$ and is not one of the edges considered.
The intervals of the other corners of $B$ have interiors disjoint from that of
$vh$, so their cover tests are unchanged, and the edges inside $U$ are
unchanged.
Each flip removes one of the edges considered and creates none, so after
finitely many flips there are none left.

Then no edge leaves $U\cup B$, and every edge created while processing $B$ ends
in $B$.
A path bypassing a corner outside $U\cup B$ cannot use such an edge, since it
could not leave $U\cup B$ again; hence such a bypass would have existed before
we processed $B$.
Thus all corners are covers again.
When all chains have been processed, all edges of the orientation lie in
$\Gplus$, and only edges outside $\Gplus$ were removed.
\end{proof}

\subsection{Exchanging two edges}\label{subs:slides}

The next lemma exchanges two edges along two crossing covers, restoring all
other edges.

\begin{lemma}\label{lem:crossingSlide}
Let $a<b<c<d$, and let $\theta$ be an orientation in which $ac$ and $bd$ are
covers, $ab\in\theta$ and $cd\notin\theta$.
Then $\theta$ is joined by flips to the orientation $\theta-ab+cd$, in which $ac$
and $bd$ are again covers.
\end{lemma}

\begin{proof}
All edges below lie inside $ac$ or $bd$, so they are natural and normal.
We perform the flips of Figure~\ref{fig:crossingSlide} in four stages:
\begin{enumerate}[(1)]
\item $av\mapsto vc\ [ac]$ for all $av\in\theta$ with $a<v<b$, in increasing order of $v$;
\item $wd\mapsto bw\ [bd]$ for all $wd\in\theta$ with $c<w<d$, in decreasing order of $w$;
\item $ab\mapsto bc\ [ac]$ and then $bc\mapsto cd\ [bd]$;
\item the flips of (2) in reverse, and then those of (1) in reverse.
\end{enumerate}
By Lemma~\ref{lem:flipCriterion}, each step is a flip as soon as its two edges
are covers before it, and then the retained cover stays a cover.
In (1), when $av$ is moved the edges from $a$ to smaller vertices have already
been moved, so $av$ is a cover; these flips only change edges with tails below
$b$, so $bd$ stays a cover.
Stage (2) is symmetric, and only changes edges with heads above $c$.
In (3), $ab$ is a cover since $a$ has no edge into $(a,b)$, and $bc\notin\theta$
since otherwise $ab,bc$ would bypass $ac$; after the first flip, $bd$ is still a
cover, since a bypass would begin with $bc$ and no path leads from $c$ to $d$.
In (4), the net change from $ab$ to $cd$ does not affect the edges inside
$[b,w]$ and $[v,c]$, so every edge is a cover when it is moved back, as after
the corresponding flip in (2) or (1).
\end{proof}

\begin{figure}[!ht]
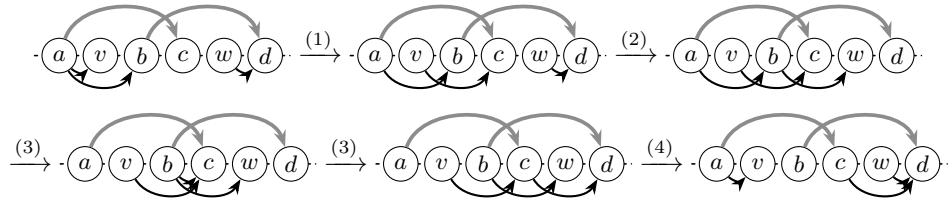

\centering
\begin{gather*}
 \EdgePic[1.55em]{6}{a,v,b,c,w,d}{1/4/c,3/6/c,1/3/b,1/2/b,5/6/b}\xrightarrow{(1)}\EdgePic[1.55em]{6}{a,v,b,c,w,d}{1/4/c,3/6/c,1/3/b,2/4/b,5/6/b}\xrightarrow{(2)}\EdgePic[1.55em]{6}{a,v,b,c,w,d}{1/4/c,3/6/c,1/3/b,2/4/b,3/5/b}\\
 \xrightarrow{(3)}\EdgePic[1.55em]{6}{a,v,b,c,w,d}{1/4/c,3/6/c,3/4/b,2/4/b,3/5/b}\xrightarrow{(3)}\EdgePic[1.55em]{6}{a,v,b,c,w,d}{1/4/c,3/6/c,4/6/b,2/4/b,3/5/b}\xrightarrow{(4)}\EdgePic[1.55em]{6}{a,v,b,c,w,d}{1/4/c,3/6/c,4/6/b,1/2/b,5/6/b}
\end{gather*}
\caption{The flips in the proof of Lemma~\ref{lem:crossingSlide}, with one
edge $av$ with $a<v<b$ and one edge $wd$ with $c<w<d$.
The gray arcs above are the covers $ac$ and $bd$, which are retained
throughout, the arcs below are the edges which move, and the dotted line stands
for all other edges, which are restored at the end.}
\label{fig:crossingSlide}
\end{figure}

We now use Lemma~\ref{lem:crossingSlide} to move edges between two corner chains.
Let $uv$ be a normal natural edge such that $u$ has a corner successor $h$.
Then $v<h$, and $vh$ is natural since it lies inside $uh$.
We call $vh$ the successor of the edge $uv$.
Passing to the successor exchanges the two values of $r$ at the endpoints, and
both endpoints increase.
Every normal edge has at most one successor and at most one predecessor, so the
normal edges form disjoint sequences in which each edge is followed by its
successor.

\begin{lemma}\label{lem:linkPaths}
For two corner chains there is at most one such sequence whose edges join the
two chains.
Along this sequence, the edges in $\Gplus$ and the edges not in $\Gplus$
alternate, and two consecutive edges in $\Gplus$ have the form $ab$ and $cd$
with $a<b<c<d$ such that $ac$ and $bd$ are corners.
\end{lemma}

\begin{proof}
Let $uv$ be the first edge of such a sequence.
If $v$ had a corner predecessor $p$, then $p<u$, since $uv$ does not contain
the corner $pv$, and then $pu$ would be a predecessor of $uv$.
Hence $v$ is the first vertex of its chain, and $u$ lies on a chain starting
earlier.
There is at most one such $u$ on a given chain: if $u<u'<v$ were two
candidates, the natural edge $uv$ would contain the corner leaving $u$.
This proves the first claim.
Passing to the successor exchanges the two values of $r$, so membership in
$\Gplus$ alternates.
If $ab$ is followed by two further edges, they use the corners leaving $a$ and $b$, say
$ac$ and $bd$, and the path continues as $ab\to bc\to cd$ with $a<b<c<d$.
\end{proof}

\begin{lemma}\label{lem:parallel}
Let $\theta,\theta'\subseteq\Gplus$ be orientations which have, for every pair
of corner chains, the same number of normal edges joining them.
Then $\theta$ and $\theta'$ lie in the same flip class.
\end{lemma}

\begin{proof}
List the positions in $\Gplus$ joining a given pair of chains in the order of
their sequence.
By Lemmas~\ref{lem:linkPaths} and~\ref{lem:crossingSlide}, we may replace an occupied
position followed by a vacant one by the reverse, restoring all other edges.
The sum of the indices of the occupied positions increases, so we reach the
configuration occupying the last positions of the list, which only depends on
the number of occupied positions.
We do this for every pair of chains, starting from $\theta$ and from $\theta'$.
\end{proof}

\begin{lemma}\label{lem:twoChains}
If there are at most two corner chains, then all orientations in which every
corner is a cover and which have the same number of ascending edges lie in one
flip class.
\end{lemma}

\begin{proof}
By Lemma~\ref{lem:preparation} we may assume that both orientations lie in
$\Gplus$.
With one chain, the only such orientation consists of the corners.
With two chains, every normal edge of $\Gplus$ joins the two chains, and
Lemma~\ref{lem:parallel} applies.
\end{proof}

\subsection{Complete graphs}\label{subs:complete}

In this subsection we label the vertices $0,1,\dotsc,n$, with $n\geq1$, and
consider the complete natural graph.
Its only corner is $0n$.
The next lemma is both the one-corner base case of
Theorem~\ref{thm:connectivity} and the local tool used later inside the last
maximal clique.

\begin{lemma}\label{lem:complete}
For the complete graph on $0,1,\dotsc,n$, the connected orientations in which
$0n$ is a cover and which have the same number of ascending edges lie in one
flip class.
\end{lemma}

\begin{proof}
Let $P_n$ consist of $0n$ and the edges $i(i+1)$ for $1\leq i<n$.
We call the other pairs $ij$, with $1\leq i<i+1<j\leq n$, the chords of $P_n$.
We first join every orientation in question to an orientation $P_n\cup X$, where
$X$ is a set of chords, and then show that the orientations $P_n\cup X$ with the
same number of chords lie in one flip class.

\smallskip\noindent\textbf{Step 1: reaching $P_n\cup X$.}
Here $r(0)=r(n)=n$ and $r(v)=v$ otherwise, so $\Gplus$ consists of $0n$ and all
edges on $1,\dotsc,n$.
By Lemma~\ref{lem:preparation} we may assume that $0n$ is the only edge at $0$.
Since $\theta$ is connected, the set $\eta$ of the remaining edges is connected on
$1,\dotsc,n$.
Whenever a vertex $i$ has two outgoing covers $ij$ and $ik$ in $\eta$, with $j<k$,
we perform the flip $ij\mapsto jk\ [ik]$.
The sum of the tails of the edges increases, so this terminates.
Then every vertex has at most one outgoing cover.
The underlying undirected graph of the covers is connected.
It contains no cycle, since the least vertex of a cycle would have two outgoing
covers.
Thus the covers form a tree with unique sink $n$.
Every edge of $\eta$ joins a vertex to one of its ancestors in this tree.

If $i<j$ are two children of $n$, their only outgoing edges are $in$ and $jn$,
and the four flips
\begin{equation}\label{eq:childMerge}
 in\mapsto0i\ [0n],\qquad jn\mapsto 0j\ [0n],\qquad
 0i\mapsto ij\ [0j],\qquad 0j\mapsto jn\ [0n]
\end{equation}
replace $in$ by $ij$, see Figure~\ref{fig:childMerge}.
For the first two, the vertex losing its edge to $n$ becomes a sink, so the new
edge from $0$ is a cover and $0n$ remains one.
At the third flip, $i$ and $j$ are sinks; afterwards $0$ is adjacent only to
the sink $j$ and to $n$, which gives the last flip.
In the tree, the branch $i\to n$ becomes $i\to j\to n$, and every edge still
joins a vertex to an ancestor.
The vertex $n-1$ is always a child of $n$.
Applying~\eqref{eq:childMerge} with $j=n-1$ until $n-1$ is the only child of
$n$, the edges on $1,\dotsc,n-1$ form a connected graph.
Then we perform $(n-1)n\mapsto 0(n-1)\ [0n]$ and apply the construction
inductively on $0,\dotsc,n-1$.
These flips remain flips in the whole graph, since a path between two vertices
below $n$ cannot pass through $n$.
At the end, $0$ is adjacent only to $n-1$ and $n$, and $n-1$ has no outgoing
edge, so we can perform $0(n-1)\mapsto(n-1)n\ [0n]$.
We have reached an orientation $P_n\cup X$.

\begin{figure}[!ht]
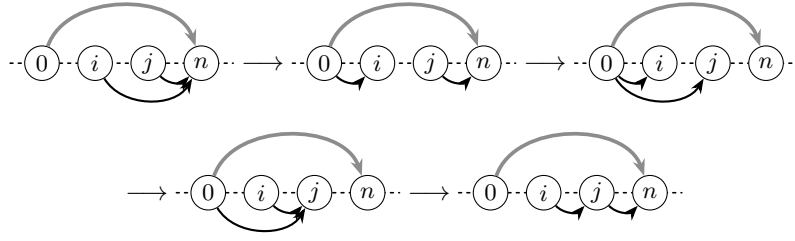

\centering
\begin{gather*}
 \EdgePic{4}{0,i,j,n}{1/4/c,2/4/b,3/4/b}\longrightarrow\EdgePic{4}{0,i,j,n}{1/4/c,1/2/b,3/4/b}\longrightarrow\EdgePic{4}{0,i,j,n}{1/4/c,1/2/b,1/3/b}\\
 \longrightarrow\EdgePic{4}{0,i,j,n}{1/4/c,2/3/b,1/3/b}\longrightarrow\EdgePic{4}{0,i,j,n}{1/4/c,2/3/b,3/4/b}
\end{gather*}
\caption{The four flips~\eqref{eq:childMerge}, which replace $in$ by $ij$; the
cover $0n$ (gray) is retained throughout.
The children $i$ and $j$ of $n$ become the chain $i\to j\to n$.}
\label{fig:childMerge}
\end{figure}

\smallskip\noindent\textbf{Step 2: working in a smaller interval.}
Let $0<p\leq m\leq n$, and consider two orientations $P_n\cup X$ and
$P_n\cup X'$ which agree on all pairs not contained in $W=\{0,p,p+1,\dotsc,m\}$.
We show that a path of flips between their restrictions to $W$, in the complete
graph on $W$, lifts to a path between them.
First we open the chain by the flips
\begin{equation}\label{eq:opening}
 j(j+1)\mapsto 0j\quad[0(j+1)],\qquad j=n-1,n-2,\dotsc,m.
\end{equation}
Before such a flip, all edges from $0$ end at $j+1$ or above, so $0(j+1)$ is a
cover, and these are flips regardless of the chords.
Now the edges at $0$ are $0m,0(m+1),\dotsc,0n$, and the restriction to $W$ is
of the same form for the complete graph on $W$.
A path between two vertices of $W$ stays in $W$: it cannot use a vertex below
$p$, since no edge from $0$ goes there, and it cannot return from above $m$.
Hence flips in the complete graph on $W$ are flips of the whole graph.
Opening $P_n\cup X'$ in the same way gives its restriction to $W$, with the same
edges outside $W$, so after the path inside $W$ we close the chain by reversing
the opening of $P_n\cup X'$.

\smallskip\noindent\textbf{Step 3: rearranging the chords.}
We show by induction on $n$ that the orientations $P_n\cup X$ with the same
number of chords lie in one flip class.
For $n\leq3$ there is at most one chord.
For $n=4$ the chords are $13$, $14$ and $24$.
In the following table, a triple $ijk$ denotes the flip exchanging $ij$ and
$jk$ across $ik$, in the direction allowed by the current state, and each row
starts and ends with $P_4$ together with the indicated chords:
\begin{center}
\begin{tabular}{@{}ll@{}}
\toprule
Chords & Flips \\
\midrule
$\{13\}\to\{14\}$ & $034,023,013,014,023,034$ \\
$\{14\}\to\{24\}$ & $034,124,014,023,012,023,034$ \\
$\{13,24\}\to\{14,24\}$ & $034,024,012,014,024,013,023,012,023,034$ \\
$\{14,24\}\to\{13,14\}$ & $034,023,012,013,023,124,034$ \\
\bottomrule
\end{tabular}
\end{center}
The cover conditions are checked directly at each step.
These rows connect all sets of one or two chords.

Let $n\geq5$.
By Step~2 and induction applied to $W=\{0,1,\dotsc,n-1\}$ and to
$W=\{0,2,3,\dotsc,n\}$, we can rearrange the chords inside either set at a fixed
number, restoring all other edges.
The two sets of chord positions cover all positions except $1n$, and both
contain $24$.
Hence we can exchange an occupied chord position with a vacant one, provided
neither is $1n$.
If no single set contains both positions, we route through $24$.
When $24$ is vacant, we move the selected chord first to $24$ and then to the
desired position.
When $24$ is occupied, we move the chord at $24$ first to the vacant position
and then move the selected chord to $24$.
This connects all orientations $P_n\cup X$ with the same number of chords and the same
status of $1n$.

It remains to change the status of $1n$.
Since the orientations with each fixed status are already connected, it is
enough to join the two statuses whenever both occur.
Suppose $1n\notin\theta$ and $13,2n\in\theta$.
We perform $in\mapsto 0i\ [0n]$ for all $in\in\theta$ with $0<i<n$, in decreasing
order of $i$.
When $in$ is moved, the edges into $n$ with larger tails have been moved.
Hence $in$ is a cover, and the vertices other than $n$ adjacent to $0$, which are
larger than $i$, cannot reach $n$, so $0n$ is a cover.
Now $02\in\theta$ and $01\notin\theta$, so $02$ is a cover, and we perform the
flip $12\mapsto 01\ [02]$.
Then we reverse these flips in increasing order of $i$, including the new edge
$01$.
When $0i$ is moved back to $in$, it has the least head among the edges at $0$,
and the edges into $n$ restored so far have smaller tails, so $0i$ and $0n$ are
covers.
We obtain $\theta-12+1n$.
To restore $12$, we open the chain down to $3$ by~\eqref{eq:opening}.
On $0,1,2,3$ the edges are then $03$, $13$ and $23$, and~\eqref{eq:childMerge}
with $(i,j,n)=(1,2,3)$ replaces $13$ by $12$.
Explicitly, its four flips are
\[
 13\mapsto01\ [03],\qquad 23\mapsto02\ [03],\qquad
 01\mapsto12\ [02],\qquad 02\mapsto23\ [03].
\]
They restore $12$ and $23$, remove $13$, and retain $03$; Step~2 carries this
local sequence out in the whole graph.
Closing the chain gives
\begin{equation}\label{eq:outerChord}
 \theta\longmapsto\theta-13+1n\qquad(13,2n\in\theta,\ 1n\notin\theta).
\end{equation}
Let $\theta=P_n\cup X$; there are $\binom{n-1}{2}$ chords in total.
If $2\leq|X|\leq\binom{n-1}{2}-1$, we can first arrange $13,2n\in\theta$ and
$1n\notin\theta$, and~\eqref{eq:outerChord} connects the two statuses of $1n$.
If $|X|=0$ or $|X|=\binom{n-1}{2}$, there is only one set of chords.
If $|X|=1$, start from $P_n+13$, open the chain down to $2$, perform
$13\mapsto01\ [03]$ and $01\mapsto 1n\ [0n]$, and close the chain.
While the chain is open, the vertices $2,\dotsc,n$ are sinks, so these are flips,
and we reach $P_n+1n$.
This completes the induction, and the proof.
\end{proof}

\subsection{Deleting a chain or an edge}

In this subsection, connectivity means connectivity of undirected graphs.
The final induction removes a corner chain or the first vertex.
The next two lemmas ensure that the remaining graph in $\Gplus$ stays connected
and that a prescribed normal edge can be avoided.
Suppose that the natural graph is connected, has at least three vertices, and
has no articulation.
Then every pair $uv$ with $v-u\leq2$ is a natural edge.
Indeed, if $u(u+1)$ were not natural, then by closure under subintervals no
natural edge would join $[1,u]$ and $[u+1,n]$; and if $u(u+2)$ were not natural,
no natural edge would pass over $u+1$, which would be an articulation.

Replacing the natural graph by $\Gplus$ preserves connectivity within any union
of corner chains.
Indeed, if $uv$ is normal with $r(u)>r(v)$, and $uh$ is the corner leaving $u$,
then $u<v<h$, the edge $vh$ lies in $\Gplus$, and the walk $u,h,v$ only visits
the chains of $u$ and $v$.

\begin{lemma}\label{lem:chainDeletion}
Deleting the vertices of one corner chain leaves the other vertices connected
in $\Gplus$.
\end{lemma}

\begin{proof}
Two consecutive vertices $v,v+1$ are not on the same chain, since their edge
would be a corner, while it lies inside $(v-1)(v+1)$ or $v(v+2)$.
Hence the remaining vertices, listed in increasing order, have differences at
most two, and consecutive ones are joined by natural edges.
We replace the edges outside $\Gplus$ by the walks above, which avoid the deleted
chain.
\end{proof}

\begin{lemma}\label{lem:removableEdge}
If there are at least three corner chains, then deleting any normal edge $uv$ of
$\Gplus$ leaves $\Gplus$ connected.
\end{lemma}

\begin{proof}
Choose a vertex $w$ on a third chain.
By Lemma~\ref{lem:chainDeletion}, there is a path from $u$ to $w$ avoiding the
chain of $v$, and a path from $w$ to $v$ avoiding the chain of $u$.
Neither uses $uv$.
\end{proof}

\subsection{The last clique}

Suppose that the last two corners are $bc$ and $an$, with $b<a<c<n$.
Let $L=[1,c]\cup\{n\}$, and let $\Gamma[L]$ be the graph induced on $L$, with the
inherited order and strict edges.
For an orientation $\theta$ we write $\theta[L]$ for its edges inside $L$.
All corners other than $an$ end at most at $c$, so the vertices in $(c,n)$ lie
on no corner.
Hence $\Gamma[L]$ has the same corners as $\Gamma$, its corner chains are those of
$\Gamma$ except for the singleton chains of the vertices in $(c,n)$, and its
set $\Gplus$ consists of the edges of $\Gplus$ inside $L$.
Every vertex strictly between $a$ and $n$ is the last vertex of its chain.
Hence an orientation in $\Gplus$ has no edge from $a$ other than $an$, and
every natural edge with both endpoints strictly between $a$ and $n$, or ending
at $n$ and starting after $a$, lies in $\Gplus$.

\begin{lemma}\label{lem:lastClique}
If Theorem~\ref{thm:connectivity} holds for $\Gamma[L]$, then it holds for
$\Gamma$.
\end{lemma}

\begin{proof}
Let $\theta$ be a connected orientation in which all corners are covers, and
assume $\theta\subseteq\Gplus$ by Lemma~\ref{lem:preparation}.
We first make the restriction to $L$ connected, so that the inductive
hypothesis can change it independently of the remaining vertices.
We then put a fixed spanning tree $T$ into the orientation.
The last step routes exchanges between the positions inside and outside $L$
through their common position $ba$; Figure~\ref{fig:lastCliqueFlips} shows the
flips used.

All pairs in $[a,n]$ are natural, and a path between two vertices of a connected
component of the edges of $\theta$ inside $[a,n]$ stays in that component.
So flips of the complete graph on such a component are flips of $\Gamma$, and
Lemma~\ref{lem:complete} applies to the component containing $a$ and $n$, since
$an$ is a cover.

\smallskip\noindent\textbf{Step 1: making $\theta[L]$ connected.}
In the component containing $a$ and $n$, the vertex $a$ is only adjacent to $n$,
and by Lemma~\ref{lem:complete} we replace its edges by the star of all edges to
$n$, together with edges avoiding $a$; this is possible since a connected graph
on $k$ vertices with a leaf has between $k-1$ and $1+\binom{k-1}{2}$ edges.
Every other component meets $(a,c]$, since otherwise it would be a connected
component of $\theta$.
Join its least vertex $u$ to another vertex $v\leq c$ by a path of covers, using
edges in either direction.
If the largest vertex $z$ of the path exceeds $v$, its neighbours $i<j$ give the
flip $jz\mapsto ij\ [iz]$ of Figure~\ref{fig:lastCliqueFlips}(a), which shortens the
path; replacing an edge which is no longer a cover by a path of covers inside its
interval, the largest vertex decreases.
We end with a path inside $[u,v]$, and these flips keep all pairs joined inside
$[1,c]$.
Every natural edge with an endpoint in $(c,n)$ lies inside $an$, so $\theta[L]$
is now connected.

\smallskip\noindent\textbf{Step 2: changing $\theta[L]$.}
The flips $vn\mapsto av\ [an]$ of Figure~\ref{fig:lastCliqueFlips}(b), for
$vn\in\theta$ with $c<v<n$ in decreasing order of $v$, are flips as in stage~(1)
of the proof of Lemma~\ref{lem:crossingSlide}, and they change neither the edges
inside $L$ nor the cover tests of other corners.
Afterwards a path leaving $L$ cannot return to $L$, so every path of flips in
$\Gamma[L]$ lifts, with all edges outside $L$ fixed.
By assumption we reach any connected orientation of $\Gamma[L]$ in $\Gplus$ with
the same number of edges, and then we undo the flips of
Figure~\ref{fig:lastCliqueFlips}(b) in increasing order of $v$.

\smallskip\noindent\textbf{Step 3: a fixed tree.}
In the set $\Gplus$ of $\Gamma[L]$, the corners and the edges $vn$ with $a<v\leq c$
form a forest, since every vertex has at most one outgoing edge among them.
It extends to a spanning tree $T'$ of $\Gamma[L]$ inside $\Gplus$.
Inside $[a,n]$, the tree $T'$ consists of edges to $n$, and $ba\notin T'$ since
$bc,cn,na$ already join $b$ and $a$.
By Step~2 we may assume that $T'\subseteq\theta$.
Then $[a,n]$ is one component: the vertices in $(c,n)$ are joined to $L$ by
unchanged edges, and those of $L\cap[a,n]$ are joined to $n$ by $T'$.
Lemma~\ref{lem:complete} makes it the star of Figure~\ref{fig:lastCliqueFlips}(c)
together with edges avoiding $a$, and we obtain an orientation containing
$T=T'\cup\{vn:c<v<n\}$, see Figure~\ref{fig:lastCliqueTree}.

\begin{figure}[!ht]
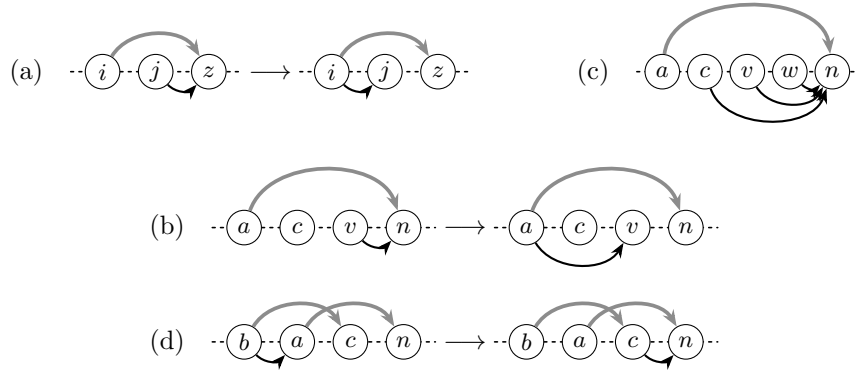

\centering
\begin{gather*}
 \text{(a)}\quad \EdgePic[2em]{3}{i,j,z}{1/3/c,2/3/b}\longrightarrow\EdgePic[2em]{3}{i,j,z}{1/3/c,1/2/b}
 \qquad\qquad
 \text{(c)}\quad \EdgePic[1.6em]{5}{a,c,v,w,n}{1/5/c,2/5/b,3/5/b,4/5/b}\\[1ex]
 \text{(b)}\quad \EdgePic[2em]{4}{a,c,v,n}{1/4/c,3/4/b}\longrightarrow\EdgePic[2em]{4}{a,c,v,n}{1/4/c,1/3/b}\\[1ex]
 \text{(d)}\quad \EdgePic[2em]{4}{b,a,c,n}{1/3/c,2/4/c,1/2/b}\longrightarrow\EdgePic[2em]{4}{b,a,c,n}{1/3/c,2/4/c,3/4/b}
\end{gather*}
\caption{The flips in the proof of Lemma~\ref{lem:lastClique}, with retained
covers in gray above the line and moving edges below.
(a) Shortening a path of covers through its largest vertex $z$ in Step~1.
(b) Moving an edge $vn$ with $c<v<n$ off $n$ in Step~2.
(c) The star at $n$ inside $[a,n]$ in Step~3.
(d) Exchanging $ba$ for $cn$ in Step~4, by Lemma~\ref{lem:crossingSlide} for the
corners $bc$ and $an$.}
\label{fig:lastCliqueFlips}
\end{figure}

\begin{figure}[!ht]
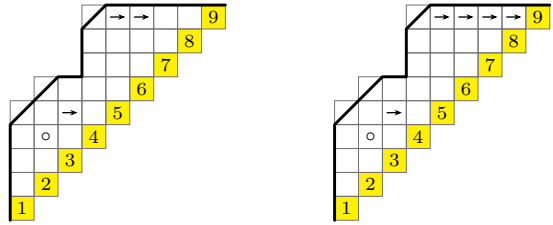

\centering
\LLTpic[cell=0.9em]{nnnnddenndeeeee}{\LLTasc{3/5,5/9,6/9}\LLTmark{2/4}{$\circ$}}\qquad\qquad
\LLTpic[cell=0.9em]{nnnnddenndeeeee}{\LLTasc{3/5,5/9,6/9,7/9,8/9}\LLTmark{2/4}{$\circ$}}
\caption{The trees $T'$ (left) and $T$ (right) for the path
$\mathtt{nnnnddenndeeeee}$, with corners $15$, $26$ and $49$, so that $b=2$, $a=4$,
$c=6$, $n=9$ and $L=[1,6]\cup\{9\}$.
The diagonal steps cross the corners, which are tree edges, and the arrows
are the remaining tree edges; $T$ adds the leaves $7$ and $8$ at $9$.
The position $ba=24$, marked $\circ$, lies in neither tree.}
\label{fig:lastCliqueTree}
\end{figure}

\smallskip\noindent\textbf{Step 4: moving the remaining edges.}
Let $M_1$ be the positions of $\Gplus\setminus T$ inside $L$, and let $M_2$ be the
set of pairs $uv$ with $a<u<v<n$, together with $ba$.
Together they are all positions of $\Gplus\setminus T$, and $ba\in M_1\cap M_2$.
Step~2 exchanges an occupied position with a vacant one in $M_1$.
Similarly, Lemma~\ref{lem:complete} on the component $[a,n]$ gives such an
exchange in $M_2\setminus\{ba\}$.
Both constructions restore all other edges.
Suppose $ba\in\theta$ and $x\in M_2\setminus\{ba\}$ is vacant.
Let $d\neq c$ be an endpoint of $x$, and let $y$ be the edge joining $c$ and $d$.
If $y\in\theta$, the component stays connected when we replace $cn$ by $x$, so
Lemma~\ref{lem:complete} performs this replacement.
Then Lemma~\ref{lem:crossingSlide} for the corners $bc$ and $an$ exchanges $ba$
for $cn$, see Figure~\ref{fig:lastCliqueFlips}(d),
and we obtain $\theta-ba+x$.
If $y\notin\theta$, the same construction with $y$ in place of $x$ gives
$\theta-ba+y$, and then we exchange $y$ and $x$.
Reversing these paths exchanges in the other direction.
Finally, a position of $M_1\setminus M_2$ and one of $M_2\setminus M_1$ are
exchanged by routing through $ba$, as in Step~3 of Lemma~\ref{lem:complete}.
Thus all orientations containing $T$ with the same number of edges lie in one flip
class, and with Lemma~\ref{lem:preparation} this proves the lemma.
\end{proof}

Reversing the order of the vertices maps flips to flips, since it reverses the
direction of all edges and preserves covers.
Hence Lemma~\ref{lem:lastClique} also holds for the first two corners, with the
graph induced on $\{1\}\cup[b',n]$ in place of $\Gamma[L]$ when the first two
corners are $1a'$ and $b'c'$.

\subsection{The first vertex}

Suppose that the first two corners are $1a$ and $2b$, with $2<a<b$.
Let $\Gamma'$ be the graph induced on $2,\dotsc,n$.
It has the corners of $\Gamma$ other than $1a$, and the same values of $r$ on
$2,\dotsc,n$.
Hence $\Gplus$ consists of $1a$, the set $(\Gamma')^{+}$ for $\Gamma'$, and the
edges $1v$ for $v$ in
\[
 J=\{v:1<v<a,\ r(a)<r(v)\}.
\]
For $v\in J$, let $vh_v$ be the corner leaving $v$.
Then $h_v>a$, since a corner starting in $(1,a)$ does not end inside $[1,a]$,
and $ah_v$ is a normal edge of $(\Gamma')^{+}$.
Lemma~\ref{lem:crossingSlide} for the corners $1a$ and $vh_v$ exchanges $1v$ and
$ah_v$:
\[
 \EdgePic{4}{1,v,a,h_v}{1/3/s,2/4/s,1/2/b}
 \;\longleftrightarrow\;
 \EdgePic{4}{1,v,a,h_v}{1/3/s,2/4/s,3/4/b}\,.
\]

\begin{lemma}\label{lem:firstVertex}
Suppose that Theorem~\ref{thm:connectivity} holds for $\Gamma'$, and that
$(\Gamma')^{+}\setminus\{ah_v\}$ is connected for every $v\in J$.
Then Theorem~\ref{thm:connectivity} holds for $\Gamma$.
\end{lemma}

\begin{proof}
Let $\theta\subseteq\Gplus$ be connected.
We first make the restriction to $[2,n]$ connected.
Once this is done, the set $I$ of edges at $1$ records the remaining choice:
the inductive hypothesis connects orientations with fixed $I$, and crossing
slides connect the possible sets $I$.
If the edges of $\theta$ inside $[2,n]$ do not form a connected graph, let $B$ be the
component containing $a$.
Some edge $1v\in\theta$ ends in another component, and then $v\in J$, the vertices
$v$ and $h_v$ lie in the same component, and $ah_v\notin\theta$.
Lemma~\ref{lem:crossingSlide} replaces $1v$ by $ah_v$, and the number of edges at $1$
decreases.
Hence we may assume that the edges inside $[2,n]$ form a connected graph.

Fix the set $I\subseteq J$ of vertices $v$ with $1v\in\theta$.
Paths between vertices of $[2,n]$ avoid $1$, so flips of $\Gamma'$ lift with the
edges at $1$ fixed, and by assumption all such orientations with the same $I$
and the same number of edges lie in one flip class.

Let $m$ be the number of edges of $(\Gamma')^{+}$, and let $s$ be the number of
edges of $\theta$.
The corners of $\Gamma'$ form a forest, which extends to a spanning tree of
$(\Gamma')^{+}$, so connected orientations of $\Gamma'$ in $(\Gamma')^{+}$ exist
with any number of edges between $n-2$ and $m$.
Hence the possible sets are exactly the $I\subseteq J$ with
\begin{equation}\label{eq:initialBand}
 \max(0,s-1-m)\leq|I|\leq\min(|J|,s-1-(n-2)).
\end{equation}
Indeed, the edges incident to $1$ are $1a$ and the edges $1v$ with $v\in I$,
so the restriction to $[2,n]$ has $s-1-|I|$ edges.
The inequalities $n-2\leq s-1-|I|\leq m$ and $0\leq|I|\leq|J|$ give
the displayed bounds.
Suppose that $I$ and $I\setminus\{v\}$ both satisfy~\eqref{eq:initialBand}.
The corresponding orientations of $\Gamma'$ have $t$ and $t+1$ edges, with
$n-2\leq t<m$.
Since $(\Gamma')^{+}\setminus\{ah_v\}$ is connected, there is a connected
orientation of $\Gamma'$ in $(\Gamma')^{+}$ with $t$ edges avoiding $ah_v$.
Together with the edges $1w$, $w\in I$, Lemma~\ref{lem:crossingSlide} moves it
to an
orientation with the set $I\setminus\{v\}$.
Thus the classes for $I$ and $I\setminus\{v\}$ agree.

Any two sets satisfying~\eqref{eq:initialBand} are joined by such insertions and
deletions: to exchange two elements we delete first if $|I|$ is above its lower
bound, and insert first otherwise.
This fails only if~\eqref{eq:initialBand} allows a single cardinality and more
than one set.
If $J\neq\emptyset$, the connected graph
$(\Gamma')^+\setminus\{ah_v\}$ has $n-1$ vertices and $m-1$ edges for every
$v\in J$, so $m\geq n-1$.
If the lower and upper bounds in~\eqref{eq:initialBand} were both an integer
$k$ with $0<k<|J|$, then
\[
 k=s-1-m=s-1-(n-2),
\]
which would give $m=n-2$, a contradiction.
Thus the only possible single cardinalities are $0$ and $|J|$, each of which
determines $I$ uniquely.
Hence all connected orientations in $\Gplus$ with $s$ edges lie in one flip
class, and Lemma~\ref{lem:preparation} proves the lemma.
\end{proof}

\subsection{Proof of Theorem~\ref{thm:connectivity}}

First let the natural graph be connected without articulations.
We proceed by induction on $n$.
If there is only one corner, the graph is complete, and Lemma~\ref{lem:complete}
applies.
If there are at most two corner chains, Lemma~\ref{lem:twoChains} applies.
Otherwise let $1a$ and $b'c'$ be the first two corners.
Since there are no articulations, consecutive maximal cliques share at least two
vertices, so $1<b'<a<c'$.
If $b'>2$, the reversed form of Lemma~\ref{lem:lastClique} reduces the claim to the
graph induced on $\{1\}\cup[b',n]$, which has fewer vertices and no
articulation.
If $b'=2$, deleting the vertex $1$ removes the corner $1a$ and gives a graph
$\Gamma'$ without articulations, whose maximal cliques start with $[2,c']$.
It has the same number of corner chains as $\Gamma$, at least three, so
Lemma~\ref{lem:removableEdge} for $\Gamma'$ verifies the hypothesis of
Lemma~\ref{lem:firstVertex}, and induction applies to $\Gamma'$.

Finally, let the natural graph consist of the path $1,\dotsc,b$ followed by a
graph without articulations on $[b,n]$.
The path edges are corners, so they lie in every orientation we consider, and
paths between vertices of $[b,n]$ do not use vertices below $b$.
Hence restriction to $[b,n]$ is a bijection between the orientations in question
of $\Gamma$ and of the graph induced on $[b,n]$, compatible with flips.
This completes the proof of Theorem~\ref{thm:connectivity}, and thereby of
Theorem~\ref{thm:main} for ordinary graphs.

\section{Flip classes of circular graphs}\label{sec:circular}

In this section we prove Theorem~\ref{thm:main} for circular graphs, by
reduction to the ordinary case.
The idea is to expand $\Kv(\theta)$ and $\Ev(\theta)$, for every admissible
orientation, as the same signed sum over \emph{balanced} sub-orientations, in
which the lengths of the edges are differences of a potential.
After summing these expansions over a fixed flip class, the resulting
coefficients are constant on flip classes of suborientations, so we may group
the expansion by those classes.
Finally, the orientations of a balanced flip class can be drawn on a line, where
they form a union of ordinary flip classes.

Let $\Gamma$ be a circular decorated unit-interval graph.
A strict edge is a corner, so it is never a short edge of a flip, and it is
present in every orientation.
Hence forgetting the strict markings, and including the strict edges in every
orientation, identifies the flip classes of $\Gamma$ with flip classes of the
undecorated graph, with the same graphs $\theta\cup S$.
We therefore assume that $S=\emptyset$, and identify an orientation $\theta$ with
the edge set of $\theta\cup S$.
Throughout, orientations are admissible, and for a set of edges $Y$ we write
$c(Y)$ for the number of connected components of the undirected graph on
$\setZ/n\setZ$ with edges $Y$, isolated vertices included.
Every edge $e$ has a \defin{length} $\ell(e)\in\{1,\dotsc,n-1\}$.

We write $\Kv(\theta)=\Kp_{\theta}$, and $h_\theta$ for the whole map
$z\mapsto\widetilde{\hrvop}(\theta,z)$ on $\setZ$.
We work in the free abelian group on such maps, with basis elements $[h_\theta]$.
The linear map $[h]\mapsto\elementaryE_\lambda$, where $\lambda$ is the partition
of the fibre sizes of $h$ modulo $n$, sends $[h_\theta]$ to $\Ev(\theta)$.

\subsection{Balanced orientations}
The signed expansion is indexed by subsets of an orientation of the following
kind.

\begin{definition}
A set $F$ of edges is \defin{balanced} if on each connected component of the
undirected graph $(\setZ/n\setZ,F)$ there is a function $p$ to $\setZ$ with
$\ell(e)=p(v)-p(u)$ for every edge $e$ from $u$ to $v$ in $F$.
\end{definition}

Equivalently, the lifted graph of $F$ is a disjoint union of translates of a
graph on finitely many integers.
Let $F\subseteq\theta$ be balanced, with such a potential $p$.
Contract the connected components of $F$, keeping the edges of $\theta\setminus F$
as individual edges, possibly loops or parallel edges, and give the edge $e$ from
$u$ to $v$ the \defin{gain} $\delta(e)=\ell(e)-(p(v)-p(u))$.
Changing $p$ by a constant on a component does not change the gain of a closed
walk.

\begin{definition}\label{def:coefficient}
Let $F\subseteq\theta$ be balanced, with contracted graph $Q$.
If every directed cycle of $Q$ has negative gain and every edge of $Q$ lies on a
directed cycle, let $\defin{\chi(\theta,F)}=(-1)^{c(F)-c(\theta)}$.
Otherwise, and for $F$ not balanced or not contained in $\theta$, let
$\chi(\theta,F)=0$.
\end{definition}

If $\theta$ itself is balanced, then $\chi(\theta,F)=[F=\theta]$.
Indeed, restricting a potential of $\theta$ to $F\subsetneq\theta$ gives every edge
of $\theta\setminus F$ gain zero, so any directed cycle containing one of them,
a loop included, has gain zero.

\begin{example}\label{ex:deletion}
Let $\avec=(2,2,2)$, so that $\Gamma$ is the complete circular graph on three
vertices, and let $\theta$ consist of the edges $1\to2$, $1\to3$ and $3\to2$, of
lengths $1$, $2$ and $2$.
It is not balanced: the path $1\to3\to2$ has length $4$, while $1\to2$ has length
$1$.
Lemma~\ref{lem:deletion} below, applied to this cycle, gives
$\Kv(\theta)=\Kv(F)$ and $[h_\theta]=[h_F]$ for $F=\{1\to3,3\to2\}$.
The set $F$ is balanced, and one component of its lift is the path
$1\to3\to5$.
Whether or not $1\to2$ is present, the highest reachable lifts of $1$, $3$ and
$2$ are $5$, $5$ and $2$, respectively.
They have the same residue modulo $3$, so $\Ev(\theta)=\Ev(F)=\elementaryE_3$.
Moreover, $F$ gives the chain $1\to3\to2$, and the extra edge $1\to2$ imposes
no additional condition on colorings.
Hence $\Kv(\theta)=\Kv(F)=\elementaryE_3$ as well.
In the lift below, the integer $z$ lies over the residue of $z$ modulo $3$, the
solid edges are the lift of $F$ starting in $\{1,2,3\}$, and the dashed edge is
the lift of the deleted edge $1\to2$:
\[
\begin{tikzpicture}[x=3.2em,y=3em]
\draw[gray] (0.4,0)--(5.6,0);
\foreach \z in {1,...,5}{\node[vtx,fill=white] (z\z) at (\z,0) {$\z$};}
\draw[sel,gray,densely dashed] (z1) to[out=-60,in=-120,looseness=1.8] (z2);
\draw[sel] (z1) to[bend left=45] (z3);
\draw[sel] (z3) to[bend left=45] (z5);
\end{tikzpicture}
\]
\end{example}

\subsection{Deleting edges of a cycle}

\begin{lemma}\label{lem:deletion}
Let $\theta$ be admissible, and let $Z$ be a cycle of the undirected graph of
$\theta$, traversed so that its forward edges form a nonempty set $Z^+$, its
backward edges form a set $Z^-$, and
\[
 \sum_{e\in Z^+}\ell(e)>\sum_{e\in Z^-}\ell(e).
\]
Let $B=\theta\setminus Z^-$.
Then
\[
 \sum_{J\subseteq Z^-}(-1)^{|J|}\Kv(B\cup J)=0
 \qquad\text{and}\qquad
 \sum_{J\subseteq Z^-}(-1)^{|J|}[h_{B\cup J}]=0.
\]
\end{lemma}

\begin{figure}[!ht]
\centering
\begin{tikzpicture}[x=3em,y=3em]
\node[vtx,fill=white] (a) at (90:1) {$1$};
\node[vtx,fill=white] (b) at (210:1) {$2$};
\node[vtx,fill=white] (c) at (330:1) {$3$};
\draw[sel] (a) to[bend left=25] node[font=\scriptsize,above right] {$2$} (c);
\draw[sel] (c) to[bend left=25] node[font=\scriptsize,below] {$2$} (b);
\draw[sel,densely dashed] (a) to[bend right=25] node[font=\scriptsize,above left] {$1$} (b);
\end{tikzpicture}
\caption{The cycle of Example~\ref{ex:deletion}, traversed $1\to3\to2\to1$, with
the edge lengths.
The solid edges form $Z^+$, of total length $4$, and the dashed edge $1\to2$
forms $Z^-$, of length $1$; Lemma~\ref{lem:deletion} removes it.}
\label{fig:deletionCycle}
\end{figure}
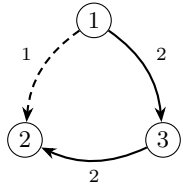

\begin{proof}
Let $\kappa$ be a coloring which increases along the edges of $B$.
The values of $\kappa$ telescope around $Z$, so
$\sum_{Z^+}(\kappa(v)-\kappa(u))=\sum_{Z^-}(\kappa(v)-\kappa(u))$, the sums ranging over
the edges from $u$ to $v$.
The left side is positive, so some edge of $Z^-$ increases.
Toggling the first such edge, in a fixed order of $Z^-$, is a sign-reversing
involution on the sets $J$ for which $\kappa$ increases along $B\cup J$, which
proves the first identity.

For the second, fix $J$ and let $h=h_{B\cup J}$.
The function $t(z)=h(z)-z$ is periodic, so it is a function on residues.
For an edge from $u$ to $v$ in $B\cup J$, its lift ends at $u+\ell$ and
$h(u)\geq h(u+\ell)$, so $t(u)-t(v)\geq\ell$.
Telescoping $t$ around $Z$ gives
\[
 \sum_{Z^-}(t(u)-t(v))=\sum_{Z^+}(t(u)-t(v))\geq\sum_{Z^+}\ell>\sum_{Z^-}\ell.
\]
Hence some edge of $Z^-$ satisfies $t(u)-t(v)>\ell$, that is, $h(u)>h(u+\ell)$ for its
lifts.
Adding this edge does not change $h$, by Lemma~\ref{lem:ineffective}(i).
Removing it does not change $h$ either: a path through a lift of this edge ends
below $h(u+\ell)<h(u)$ and is not maximizing.
Toggling the first edge of $Z^-$ with this property is therefore a sign-reversing
involution which preserves $h$.
\end{proof}

If $\theta$ is not balanced, choose a spanning forest of $\theta$ and a potential
on it.
Some edge outside the forest violates the potential, and together with the path
in the forest between its endpoints it forms a cycle with nonzero signed length.
We traverse this cycle in the direction in which the forward edges are longer.
Then $Z^-\neq\emptyset$, since otherwise $\theta$ would contain a directed cycle.
Hence every set $B\cup J$ with $J\subsetneq Z^-$ in Lemma~\ref{lem:deletion} has fewer
edges than $\theta$.

\subsection{Expansion into balanced orientations}

An \defin{ordered partition} of the vertex set of a directed multigraph $Q$ is
a sequence $\omega=(\omega_1,\dotsc,\omega_k)$ of disjoint nonempty blocks
covering it; it is \defin{compatible} with $Q$ if every edge goes from some
$\omega_i$ to some $\omega_j$ with $i\leq j$.

\begin{lemma}\label{lem:orderedPartitions}
Let $Q$ be a finite directed multigraph, and let $\omega$ range over the ordered
partitions compatible with $Q$, with $k(\omega)$ blocks.
Then $\sum_{\omega}(-1)^{k(\omega)}$ is $(-1)^{c(Q)}$ if every edge of $Q$ lies on
a directed cycle, and $0$ otherwise.
\end{lemma}

\begin{proof}
A strongly connected component lies in one block of a compatible ordered
partition, so we may contract the strong components and remove the loops.
The edges of $Q$ all lie on directed cycles exactly when the resulting acyclic
graph $Q'$ has no edges, and in that case $c(Q)$ is the number of vertices of $Q'$.
We show by induction that the sum for $Q'$ is $(-1)^{|Q'|}$ if $Q'$ has no edges and
$0$ otherwise.
The last block $U$ of a compatible ordered partition is a nonempty set closed
under taking heads of edges, so the sum for $Q'$ is minus the sum over such $U$ of
the sums for $Q'\setminus U$.
By induction only the sets $U$ containing every head of an edge contribute, and
the sum becomes $-(-1)^{|Q'|}\sum_U(-1)^{|U|}$ over the nonempty sets $U$
containing all heads.
If $Q'$ has an edge, it has a source which is not a head, and the inner sum
vanishes.
If $Q'$ is nonempty and has no edges, the inner sum is $-1$, so the full
expression is $(-1)^{|Q'|}$.
\end{proof}

\begin{proposition}\label{prop:balancedExpansion}
For every admissible orientation $\theta$,
\[
 \Kv(\theta)=\sum_{F\subseteq\theta}\chi(\theta,F)\,\Kv(F)
 \qquad\text{and}\qquad
 [h_\theta]=\sum_{F\subseteq\theta}\chi(\theta,F)\,[h_F].
\]
In particular $\Ev(\theta)=\sum_F\chi(\theta,F)\,\Ev(F)$.
\end{proposition}

\begin{proof}
We proceed by induction on $|\theta|$.
If $\theta$ is balanced, then $\chi(\theta,F)=[F=\theta]$.
Otherwise choose a cycle as after Lemma~\ref{lem:deletion}.
That lemma expresses both values of $\theta$ through the sets $B\cup J$ with
$J\subsetneq Z^-$, and by induction it suffices to prove
\begin{equation}\label{eq:coefficientCancellation}
 \sum_{J\subseteq Z^-}(-1)^{|J|}\chi(B\cup J,F)=0
 \qquad\text{for all } F\subseteq\theta.
\end{equation}

Fix $F$; we may assume that $F$ is balanced, with a potential $p$.
Contract the components of $F$ and let $Q$ be the resulting multigraph with gains.
Its edges are the edges of $B\setminus F$, which we call required, and the
edges of $Z^-\setminus F$, which we call optional.
Only the sets $J$ containing $F\cap Z^-$ contribute.
Write $Q(J)$ for the graph of the required edges together with the optional
edges in $J$.
Since $c(B\cup J)=c(Q(J))$, the identity~\eqref{eq:coefficientCancellation}
says, up to a common sign, that the sum of $(-1)^{|J|+c(Q(J))}$ over the sets
$J\subseteq Z^-\setminus F$ for which every directed cycle of $Q(J)$ has negative
gain and every edge of $Q(J)$ lies on a directed cycle vanishes.
The cycle $Z$ projects to a closed walk in $Q$ which follows the required edges
forwards and the optional edges backwards, and since the potential
telescopes, its gain satisfies
\begin{equation}\label{eq:walkGain}
 \sum_{\text{required}}\delta-\sum_{\text{optional}}\delta
 =\sum_{Z^+}\ell-\sum_{Z^-}\ell>0,
\end{equation}
the sums on the left ranging over the edges of the walk.

By Lemma~\ref{lem:orderedPartitions}, this sum equals the sum over ordered
partitions $\omega$ of $(-1)^{k(\omega)}$ times the sum of $(-1)^{|J|}$ over the
sets $J$ such that $\omega$ is compatible with $Q(J)$ and every directed cycle of
$Q(J)$ has negative gain.
Fix $\omega$, and assume that the required edges are compatible with $\omega$,
since otherwise the inner sum is empty.
If an optional edge compatible with $\omega$ goes strictly upwards between blocks,
it lies on no cycle of a graph compatible with $\omega$, and toggling it is a
sign-reversing involution.
Otherwise every compatible optional edge stays in one block.
Along the closed walk, required edges do not decrease the block index,
compatible optional edges traversed backwards keep it, and incompatible optional
edges traversed backwards increase it strictly.
The total change is zero, so the walk lies in a single block, and it remains to
show that the sum of $(-1)^{|J|}$ over the sets $J$ of compatible optional edges
for which every directed cycle of $Q(J)$ has negative gain vanishes.

For such $J$, the gains of directed walks in $Q(J)$ ending at a vertex $y$ are
bounded, and we let $\tau(y)$ be the maximum of $0$ and these gains.
Then $\tau(v)-\tau(u)\geq\delta(e)$ for every edge $e$ from $u$ to $v$ of $Q(J)$.
Summing $\tau(v)-\tau(u)$ around the closed walk gives zero, so
by~\eqref{eq:walkGain} some optional edge $e$ of the walk, from $u$ to $v$, has
$\tau(v)-\tau(u)>\delta(e)$.
A directed cycle through such an edge has gain at most $\delta(e)+\tau(u)-\tau(v)<0$,
and every directed walk ending at a vertex $y$ and containing $e$ has gain
strictly less than $\tau(y)$.
Hence adding or removing $e$ keeps all directed cycles negative and does not change
$\tau$.
Toggling the first such edge, in a fixed order of the optional edges, is
therefore a sign-reversing involution, which proves the proposition.
\end{proof}

\subsection{Constancy on flip classes}

Let $C$ be a flip class of $\Gamma$, and for an admissible orientation $\eta$ put
\[
 a_C(\eta)=\sum_{\alpha\in C}\chi(\alpha,\eta).
\]

\begin{lemma}\label{lem:constancy}
The function $a_C$ is constant on every flip class.
\end{lemma}

\begin{proof}
Let $\theta$ and $\eta$ differ by a flip which exchanges $ab$ for $bc$ across
$ac$, in a window $a<b<c<a+n$ of the lift.
Let $\alpha\in C$ with $\chi(\alpha,\theta)\neq0$, so $\theta\subseteq\alpha$.
Choose a potential $p$ of $\theta$ and contract the components of $\theta$.
In the contracted graph, every directed cycle formed by edges of
$\alpha\setminus\theta$ has negative gain, and since gains are integers, it is at
most $-1$.
If $m$ is the number of contracted vertices, the gains $\delta(e)+1/(m+1)$ still
give every such cycle negative gain: a simple directed cycle has at most $m$
edges, and
\[
 -1+\frac{m}{m+1}<0.
\]
As in the proof of
Proposition~\ref{prop:balancedExpansion} there is a function $\tau$ on the
contracted vertices with $\tau(v)-\tau(u)\geq\delta(e)+1/(m+1)$ for these edges.
Then $z(v)=p(v)+\tau(v)$ satisfies $z(v)-z(u)=\ell(e)$ for the edges of $\theta$ and
$z(v)-z(u)>\ell(e)$ for the edges of $\alpha\setminus\theta$.
Since $ab,ac\in\theta$ and $ac$ is a cover, we have $bc\notin\theta$.
Moreover,
\[
 z(c)-z(b)=\ell(ac)-\ell(ab)=\ell(bc),
\]
so $bc$ cannot belong to $\alpha\setminus\theta$, where the inequality for $z$
is strict.
Thus $bc\notin\alpha$.
Let $\alpha'=\alpha-ab+bc$.

A lifted path in $\alpha$ between the endpoints of $ab$ has total length
$\ell(ab)$ and total increase of $z$ equal to $\ell(ab)$, so all its edges satisfy
the equality, that is, they lie in $\theta$.
Since $ab$ and $ac$ are covers for $\theta$, they are covers for $\alpha$, and
similarly $bc$ and $ac$ are covers for $\alpha'$, using $\eta$.
The function $z$ increases along every edge of $\alpha'$, so $\alpha'$ is admissible,
and $\alpha'$ differs from $\alpha$ by a flip; hence $\alpha'\in C$.
The orientations $\theta$ and $\eta$ have the same components and the same
potential, and $\alpha'\setminus\eta=\alpha\setminus\theta$, so
$\chi(\alpha',\eta)=\chi(\alpha,\theta)$.
The same construction from $\eta$ inverts $\alpha\mapsto\alpha'$, which proves
$a_C(\theta)=a_C(\eta)$.
\end{proof}

Summing Proposition~\ref{prop:balancedExpansion} over $C$ and grouping by flip
classes $C'$, we obtain
\begin{equation}\label{eq:grouping}
 \Kv(C)=\sum_{C'} a_C(C')\,\Kv(C'),\qquad \Ev(C)=\sum_{C'} a_C(C')\,\Ev(C'),
\end{equation}
where $a_C(C')$ denotes the common value on $C'$, and $a_C(C')=0$ unless $C'$
contains a balanced orientation.

\subsection{Balanced classes are ordinary}

\begin{proposition}\label{prop:balancedClass}
Let $C$ be a flip class containing a balanced orientation.
Then $\Kv(C)=\Ev(C)$.
\end{proposition}

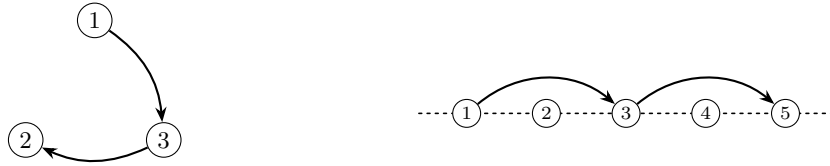
\begin{figure}[!ht]
\centering
\begin{tikzpicture}[x=3em,y=3em]
\node[vtx,fill=white] (a) at (90:1) {$1$};
\node[vtx,fill=white] (b) at (210:1) {$2$};
\node[vtx,fill=white] (c) at (330:1) {$3$};
\draw[sel] (a) to[bend left=25] (c);
\draw[sel] (c) to[bend left=25] (b);
\begin{scope}[xshift=11em,yshift=-0.5em]
\draw[llt arbitrary,line width=0.6pt] (0.4,0)--(5.6,0);
\foreach \z in {1,...,5}{\node[vtx,fill=white,minimum size=1.05em,font=\scriptsize]
  (z\z) at (\z,0) {$\z$};}
\draw[sel] (z1) to[bend left=40] (z3);
\draw[sel] (z3) to[bend left=40] (z5);
\end{scope}
\end{tikzpicture}
\caption{The balanced orientation $F=\{1\to3,3\to2\}$ of
Example~\ref{ex:deletion}, with potential values $p(1)=1$, $p(3)=3$ and
$p(2)=5$.
These values select the lifts $1$, $3$ and $5$, which form the path
$1\to3\to5$.
Ordering the original vertices by these values gives an ordinary path on three
vertices, with both edges ascending.}
\label{fig:balancedLine}
\end{figure}

\begin{proof}
Let $p$ be a potential of a balanced orientation in $C$, see
Figure~\ref{fig:balancedLine}.
On each component, $p(v)-v$ is constant modulo $n$, so changing $p$ by constants
we may assume that $p(v)\equiv v\pmod n$ and that $p$ is injective.
The relation $\ell(ac)=\ell(ab)+\ell(bc)$ shows that $p$ is a potential for every
orientation in $C$.
Order the vertices by the values of $p$, and let $\Gamma'$ be the ordinary graph on
these ranks in which a pair is a natural edge if the corresponding lifts
$p(u)<p(v)$ form a lifted natural edge.
It is a unit-interval graph since the lifted natural edges are closed under
subintervals.

An orientation $\xi$ of $\Gamma'$ maps to the circular orientation consisting of the
residues of its lifted edges.
A lifted path starting at $p(u)$ only visits the chosen lifts $p(w)$, since
$p$ is injective modulo $n$.
Hence the image is admissible, its covers correspond to those of
$\Pplus{\xi}$, its colorings correspond to those of $\xi$, and its highest
reachable lifts are the values $p(\hrv{\xi}{u})$.
So the map preserves $\Kv$ and $\Ev$, and maps ordinary flips to circular flips;
the window has length less than $n$ since the long edge is natural.
Every orientation in $C$ is the image of some $\xi$, since $p$ is a potential for
it.
Hence the set of $\xi$ whose image lies in $C$ is closed under ordinary flips,
and is a union of ordinary flip classes, on which the ordinary case of
Theorem~\ref{thm:main} gives $\Kv=\Ev$.
\end{proof}

\begin{proof}[Proof of Theorem~\ref{thm:main} for circular graphs]
By~\eqref{eq:grouping} and Proposition~\ref{prop:balancedClass},
$\Kv(C)=\sum_{C'}a_C(C')\Kv(C')=\sum_{C'}a_C(C')\Ev(C')=\Ev(C)$.
\end{proof}

\section{Consequences}\label{sec:consequences}

\subsection{The elementary expansion}

\begin{proof}[Proof of Corollary~\ref{cor:circularExpansion}]
By Lemma~\ref{lem:stateExpansion}, and since $\asc$ is constant on flip classes,
\[
 \LLT_\Gamma(\xvec;q+1)=\sum_{\theta}q^{\asc(\theta)}\Kp_{\theta}(\xvec)
 =\sum_Cq^{\asc(C)}\Kv(C),
\]
the second sum ranging over the flip classes.
By Theorem~\ref{thm:main} this equals
\[
 \sum_Cq^{\asc(C)}\Ev(C)=\sum_\theta q^{\asc(\theta)}\elementaryE_{\lambda(\theta)}(\xvec).
\qedhere
\]
\end{proof}

Thus the coefficient of $q^k\elementaryE_\lambda$ in $\LLT_\Gamma(\xvec;q+1)$ is the
number of admissible orientations $\theta$ with $\asc(\theta)=k$ and
$\lambda(\theta)=\lambda$, and every flip class contributes a sum of such terms.
Substituting $q-1$ for $q$, we also have
$\LLT_\Gamma(\xvec;q)=\sum_\theta(q-1)^{\asc(\theta)}\elementaryE_{\lambda(\theta)}(\xvec)$,
which is the form in which the expansion was stated
in~\cite{AlexanderssonSulzgruber2022}.
For ordinary graphs, Corollary~\ref{cor:circularExpansion} is
Theorem~\ref{thm:AS}, and the proof above does not use the recursion
of~\cite{AlexanderssonSulzgruber2022}.
For circular graphs it answers the question
of~\cite[Sec.~6.8]{AlexanderssonSulzgruber2022}.

\subsection{LLT cumulants}

We now refine the elementary expansion of LLT cumulants of M.~Dołęga and
M.~Kowalski~\cite{DolegaKowalski2021x} to flip classes, for circular graphs.
Let $\Gamma$ be a circular decorated unit-interval graph and $U$ a nonempty set of
residues.
List the lifts of the elements of $U$ increasingly as $(u_j)_{j\in\setZ}$, so that
$u_{j+m}=u_j+n$ with $m=|U|$, and put $b'(j)=\min\{k:u_k\geq b(u_j)\}$.
Then $b'$ is weakly increasing, $j-m<b'(j)\leq j$, and $b'(j+m)=b'(j)+m$, so it
defines a circular unit-interval graph on $U$ whose lifted natural edges are the
pairs $u_ku_j$ with $b(u_j)\leq u_k<u_j$.
A strict edge $u_ku_j$ of $\Gamma$ remains a corner, since $b(u_j)=u_k<b(u_j+1)$ gives
$b'(j)=k<b'(j+1)$.
We call the resulting graph the \defin{induced} graph $\Gamma[U]$.
Its natural edges are exactly the natural edges of $\Gamma$ with both endpoints in
$U$, and its strict edges are the strict edges of $\Gamma$ inside $U$.

Let $r\geq1$, and let $f:\setZ/n\setZ\to[r]$ be surjective and satisfy
$f(u)=f(v)$ for every strict edge $uv$.
For nonempty $B\subseteq[r]$ write $\Gamma[B]$ for the graph induced on
$f^{-1}(B)$.
Let $\Pi_r$ be the lattice of set partitions of $[r]$, with maximum $\hat1=\{[r]\}$.
Following~\cite{DolegaKowalski2021x}, the \defin{LLT cumulant} is
\[
 \defin{\kappa_{\Gamma,f}(\xvec;q)}\coloneqq(q-1)^{1-r}\sum_{\pi\in\Pi_r}
 (-1)^{|\pi|-1}(|\pi|-1)!\prod_{B\in\pi}\LLT_{\Gamma[B]}(\xvec;q).
\]
For an admissible orientation $\theta$, let $\sigma_f(\theta)$ be the set partition of
$[r]$ into the connected components of the undirected multigraph on $[r]$ with an
edge $f(u)f(v)$ for every edge $uv$ of $\theta\cup S$.
We call $\theta$ \defin{$f$-connected} if $\sigma_f(\theta)=\hat1$.

\begin{proposition}[{cf.~\cite[Thm.~3.8]{DolegaKowalski2021x}}]\label{prop:cumulant}
We have
\[
 \kappa_{\Gamma,f}(\xvec;q+1)=\sum_{\theta}q^{\asc(\theta)+1-r}\,\Kp_{\theta}(\xvec),
\]
the sum ranging over the $f$-connected admissible orientations of $\Gamma$, and all
exponents are nonnegative.
\end{proposition}

\begin{proof}
For $\pi\in\Pi_r$, the orientations of $\Gamma$ without ascending edges between
the sets $f^{-1}(B)$, $B\in\pi$, are the tuples of orientations of the graphs
$\Gamma[B]$, since no strict edge joins two such sets.
By Lemma~\ref{lem:stateExpansion} for each $\Gamma[B]$,
\[
 \prod_{B\in\pi}\LLT_{\Gamma[B]}(\xvec;q+1)
 =\sum_{\sigma_f(\theta)\leq\pi}q^{\asc(\theta)}\Kp_{\theta}(\xvec).
\]
Inserting this in the definition, the coefficient of a fixed $\theta$ is
$q^{\asc(\theta)+1-r}$ times
$\sum_{\sigma_f(\theta)\leq\pi}(-1)^{|\pi|-1}(|\pi|-1)!$.
Since $(-1)^{|\pi|-1}(|\pi|-1)!$ is the Möbius function of $\Pi_r$ from $\pi$ to
$\hat1$, this sum is $[\sigma_f(\theta)=\hat1]$.
Since $f(u)=f(v)$ for every strict edge $uv$, strict edges become loops in the
multigraph on $[r]$.
They do not connect different vertices, so an $f$-connected orientation has at
least $r-1$ ascending edges.
\end{proof}

By Lemma~\ref{lem:flipInvariants}, being $f$-connected is constant on flip classes.
Grouping Proposition~\ref{prop:cumulant} by flip classes and applying
Theorem~\ref{thm:main}, we obtain the following.

\begin{corollary}\label{cor:cumulant}
With the notation above,
\[
 \kappa_{\Gamma,f}(\xvec;q+1)=\sum_Cq^{\asc(C)+1-r}\sum_{\theta\in C}
 \elementaryE_{\lambda(\theta)}(\xvec),
\]
where $C$ ranges over the flip classes of $f$-connected admissible orientations.
In particular, every flip class contributes an elementary-positive summand.
\end{corollary}

For ordinary graphs, summing over the classes recovers the elementary expansion
of~\cite[Thm.~3.17]{DolegaKowalski2021x}.

\subsection{Questions}

For a unit-interval graph without strict edges, the chromatic quasisymmetric
function and the unicellular LLT polynomial are related by a
plethysm, see~\cite[Lem.~6.17]{AlexanderssonSulzgruber2022}.
OpenAI proved that the elementary coefficients of the full chromatic
quasisymmetric function belong to
$\mathbb{N}[q]$~\cite[Thm.~1.1]{OpenAI2026}.
B.-H.~Hwang's stronger classwise conjecture remains open: each flip class of
acyclic orientations should contribute an elementary-positive
summand~\cite[Conj.~3.10]{Hwang2022x}.
Theorem~\ref{thm:main} therefore establishes the LLT classwise analogue before
its chromatic counterpart.

The two kinds of flip give different components.
Hwang reverses a source or sink fork whose outer vertices are nonadjacent,
whereas Definition~\ref{def:flip} reverses the two short edges of a triangle and
retains its long edge.
Moreover, Hwang's classes consist of acyclic orientations, while our classes
range over all orientations of the normal edges.
The plethystic relation does not preserve elementary positivity term by term:
already $(q-1)^{-1}\elementaryE_2[(q-1)\xvec]=(q+1)\elementaryE_2-\elementaryE_{11}$.

\begin{question}
Is there a direct relation between Hwang's flip classes and the flip classes of
Definition~\ref{def:flip} which transfers classwise elementary positivity
between the chromatic and LLT settings?
\end{question}

\begin{question}
Give a bijective proof of Theorem~\ref{thm:main}, that is, a bijection between the
strict $\Pplus{\theta}$-partitions for $\theta\in C$ and the strict partitions of the
chains of the highest reachable fibres for $\theta\in C$, preserving the multiset of
colors.
\end{question}

\subsection*{Acknowledgements}
The first named author thanks Robin Sulzgruber for the initial discussions about
flips.

\subsection*{AI assistance}
We used OpenAI Codex and Claude for proof development, Lean formalization, and
preparation of the text and diagrams.
Theorem~A and the elementary expansions have been verified in Lean.
The authors are responsible for the final mathematical content and exposition.

\bibliographystyle{amsalpha}
\bibliography{references}

\end{document}